\documentclass[pdflatex,sn-mathphys-num]{sn-jnl}

\usepackage{graphicx}%
\usepackage{multirow}%
\usepackage{amsthm}%
\usepackage{amsmath,amssymb,amsfonts}%
\usepackage{mathrsfs}%
\usepackage[title]{appendix}%
\usepackage{xcolor}%
\usepackage{textcomp}%
\usepackage{manyfoot}%
\usepackage{booktabs}%
\usepackage{algorithm}%
\usepackage{algorithmicx}%
\usepackage{algpseudocode}%
\usepackage{listings}%
\usepackage{csquotes}
\usepackage{stmaryrd}
\usepackage{placeins}
\usepackage{hyperref}
\usepackage{bookmark}
\newcommand{\bo}{\boldsymbol}

\newcommand{\OO}{\Omega}
\newcommand{\dd}{{\rm d}}
\newcommand{\bs}{\boldsymbol \sigma}
\newcommand{\bt}{\boldsymbol \tau}
\newcommand{\ddiv}{\operatorname{div}}
\newcommand{\spn}{\operatorname{span}}

\newcommand{\E}{\mathcal E}
\newcommand{\R}{\mathbb R}
\newcommand{\T}{\mathcal T}

\newcommand{\Sp}{\operatorname{span}}

\theoremstyle{thmstyleone}%
\newtheorem{lemma}{Lemma}
\newtheorem{theorem}{Theorem}
\newtheorem{corollary}{Corollary}
\theoremstyle{thmstyletwo}%
\newtheorem{remark}{Remark}%

\theoremstyle{thmstylethree}%

\begin{document}

\title[Refined convergence structures of the rectangular Raviart-Thomas element]{Refined convergence structures of the rectangular Raviart-Thomas element}


\author[1,2]{\fnm{Yifan} \sur{Yue}}\email{yueyifan@stu.xmu.edu.cn}

\author*[1,2]{\fnm{Hongtao} \sur{Chen}}\email{chenht@xmu.edu.cn}

\author[3,4]{\fnm{Shuo} \sur{Zhang}}\email{szhang@lsec.cc.ac.cn}

\affil*[1]{\orgdiv{School of Mathematical Sciences}, \orgname{Xiamen University}, \orgaddress{\city{Xiamen}, \postcode{361005}, \country{China}}}

\affil[2]{\orgdiv{Fujian Provincial Key Laboratory of Mathematical Modeling and High-Performance Scientific Computing}, \orgname{Xiamen University}, \orgaddress{\city{Xiamen}, \postcode{361005}, \country{China}}}

\affil[3]{\orgdiv{LSEC, Institute of Computational Mathematics and Scientific/Engineering Computation, Academy of Mathematics and System Science}, \orgname{ Chinese Academy of Sciences}, \orgaddress{\city{Beijing}, \postcode{100190}, \country{China}}}

\affil[4]{\orgdiv{School of Mathematical Sciences}, \orgname{Chinese Academy of Sciences}, \orgaddress{\city{Beijing}, \postcode{100049}, \country{China}}}


\abstract{In this work, we fully explore three refined convergence structures of the lowest-order rectangular Raviart-Thomas element in solving the Laplace eigenvalue problem. Firstly, the scheme possesses a property of supercloseness between the discrete eigenfunctions and the interpolated ones, so that post-processing can be easily constructed to improve the accuracy at most by one order. The essentially skillful method is the integral expansion for interpolation terms. Secondly, based on the supercloseness property, we derive the error expansions for not only simple eigenvalues but also multiple eigenvalues, and provide a rigorous proof for them, based on which Richardson extrapolation can be performed. As a byproduct, we prove that all eigenvalues converge from above. Moreover, by utilizing the supercloseness property and Rayleigh quotient analysis, we give a rigorous proof for the convergence behavior for multiple eigenvalues on uniform meshes for the problem on the square domain. Thirdly, the equivalence between the lowest-order rectangular Raviart-Thomas element and the enriched rotated bilinear element is also indicated.  At the last of this work, several numerical experiments are designed to demonstrate our theory.}




\maketitle

\section{Introduction}\label{sec1}

Eigenvalue problems play an important role in engineering society \cite{babuska}. Many eigenvalue problems and corresponding Garlekin methods were discussed by Babu\v{s}ka and Osborn in a series of work \cite{babuska,babuska87,babuska89}, in which a universal framework for convergence structures was constructed. For the numerical computation of eigenvalue problems by finite elements, refined convergence structures are concerned. Some scholars are interested in the eigenvalue bounds \cite{armentano2004asymptotic,Hu.J;Huang.Y;Lin.Q2014,Li.Y2008}, which help estimate the location interval of eigenvalues. Other scholars show their interest in the accuracy of eigenfunctions and eigenvalues \cite{chen11,linxie,linxie12,sheng21,yang13}, for example, superconvergence and extrapolation.

In this work, the model problem is the Laplace eigenvalue problem with Dirichlet boundary condition. For the numerical method, we consider the lowest-order rectangular Raviart-Thomas (RRT) element, and solve three main issues of refined convergence structures: supercloseness of eigenfunctions, error expansions for eigenvalues and the equivalence of finite elements.

The RRT element was first proposed along with triangular Raviart-Thomas element in \cite{rt} in 1977. There has been lots of literature studying the superconvergence and extrapolation of mixed schemes for second order elliptic problems \cite{brand94,chen11,duran90,douglas89,fair06,lin06,linxie,linxie12,sheng21,yang13}. For the lowest-order triangular Raviart-Thomas element for second order elliptic problems, \cite{brand94} studies the superconvergence and a posteriori error estimation, \cite{duran99} considers a posteriori error estimators for the Laplace eigenvalue problem, and \cite{chen11} provides the superconvergence property and postprocessing for eigenfunctions and eigenvalues. For the lowest-order RRT element, \cite{duran90} studies the superconvergence property on global domain and at Gaussian points, \cite{fair06} proposes the error expansions for solutions based on which an extrapolation rule is proposed, while \cite{linxie} considers eigenvalue problem, and proves a suboptimal result for Richardson extrapolation, but \cite{linxie} does not consider the superconvergence of eigenfunctions. Some other work provides framework on different elements, for example, \cite{douglas89} gives the superconvergence along Gauss lines for Brezzi-Douglas-Fortin-Marini (BDFM) element, \cite{linxie12} provides a framework of superconvergence for a second order elliptic eigenvalue problem for general mixed elements with commuting diagram property, and \cite{yang13} studies the error expansion and Richardson extrapolation for eigenvalues computed by the extended rotated bilinear (${\rm E}Q_1^{\operatorname{rot}}$) element. For the first issue, we consider the supercloseness of Laplacian eigenfunctions by error expansions and propose postprocessing interpolations for better accuracy.

Lin et al. \cite{lin06} propose a universal framework for the error expansions of Laplacian eigenvalues and biharmonic eigenvalues. By exploiting the Bramble-Hilbert Lemma, they provide an available methodology for expanding the integral terms with interpolation errors. Several finite elements for Laplacian eigenvalues are considered, including triangular linear ($P_1$) element, rectangular bilinear ($Q_1$) element, rotated bilinear ($Q_1^{\operatorname{rot}}$) element, ${\rm E}Q_1^{\operatorname{rot}}$ element and high-order polynomial elements. Some integral expansion of mixed elements are also considered. However, there still remain some unsolved problems. Compared to \cite{linxie}, for the mixed problem, what is the optimal error expansions of the lowest-order RRT element? To solve this problem, we derive a new expansion for mixed elements, and improve the convergence rate of the residual of integral expansion in \cite{lin06} to the optimal. The supercloseness property is introduced for a delicate expansion for eigenvalues. As byproducts, we prove that the numerical eigenvalues by lowest-order RRT converge to exact ones from above for general rectangular meshes on rectangular domains, and provide a lower bound for the error of eigenvalues. Moreover, for the problem on the square domain, we prove the behavior of multiple eigenvalues on uniform meshes that eigenfunctions tend to distinguish from each other by their \enquote{frequencies}.

In 2015, Hu and Ma \cite{HuMa15} prove the equivalence between the enriched Crouzeix-Raviart and Raviart-Thomas elements on triangular meshes. Their work dates back to the pioneer work by Arnold and Brezzi \cite{arnold85} in 1985, and Marini \cite{marini85} in the same year. In 2026, the equivalence property on simplex meshes in two dimensions is further indicated as symmetry between the Whitney form complex and its nonconforming dual complex in $n$ dimensions by Zhang \cite{zhang26}, in the perspective of finite element exterior calculus. The ${\rm E}Q_1^{\operatorname{rot}}$ and RRT elements can be viewed as the counterparts of the enriched Crouzeix-Raviart and Raviart-Thomas elements on rectangular meshes, respectively. Inspired by the above work, we prove the equivalence between the ${\rm E}Q_1^{\rm rot}$ and RRT elements. As the third issue of our work, the equivalence property provides a perspective for reviewing the finite elements and extends the applicability of our results.

The outline of this work is organized as follows. In Section \ref{sec:pre}, we introduce some necessary notations, basis of RRT element and the spectral theory, including the errors of interpolation, eigenvalues and eigenspaces. In Section \ref{sec:spc}, we derive the supercloseness property for the eigenvalue problem, and propose the superconvergence for eigenfunctions after postprocessing. In section \ref{sec:exp}, we derive the expansion for Laplacian eigenvalues, improve the convergence rate of residual to the optimal, and derive a rigorous proof for the convergence behavior for multiple eigenvalues of the square domain on uniform meshes. At the end of this section, we provide the extrapolation formula for eigenvalues. Besides, we obtain the upper bound property of discrete eigenvalues and provide a lower bound for the error of eigenvalues. In Section \ref{sec:equ}, we prove the equivalence between the projected ${\rm E}Q_1^{\rm rot}$ element and RRT elements for both the Poisson equation and the Laplace eigenvalue problem. In Section \ref{sec:num}, we show our numerical experiments to verify the validity of our theoretical results.

\section{Preliminaries}
\label{sec:pre}

In this section, we introduce notations and the model problem. Let $\Omega$ be the considered connected domain which can be divided into non-overlapping rectangles without hanging points. We use usual notations $\nabla, \ddiv, \Delta$ to denote the gradient, divergence and Laplace operators, respectively. Throughout the paper, $L^2(\OO)$ and $H^s(\OO)~(s>0)$ are the usual Sobolev spaces on $\OO$ endowed with norms $\|\cdot\|_{0,\OO}$ and $\|\cdot\|_{s,\OO}$, respectively. Generally, $|\cdot|_{s,\OO}$ is the s-seminorm on $H^s(\OO)$. The $L^2$ inner product on $\OO$ is denoted by $(\cdot,\cdot)_{\OO}$. We omit the subscript $\OO$ in the notations of norms, seminorms and inner product if no ambiguity is introduced. The Sobolev space $\bo{H}(\ddiv,\OO)$ is defined as
$$\bo{H}(\ddiv,\OO):=\{\bt\in \bo{L}^2(\OO):\ddiv\bt\in L^2(\OO)\}$$
endowed with norm $\|\cdot\|_{\ddiv}$ with $\|\bt\|_{\ddiv}^2 = \|\bt\|_0^2 + \|\ddiv\bt\|_0^2$ for all $\bt\in\bo{H}(\ddiv,\OO)$. We use \textbf{bold} characters for vector-valued quantities hereinafter.

The Laplace eigenvalue problem with Dirichlet boundary condition seeks $(\lambda,u)\in \R\times H^2(\OO)$ with $\|u\|_0=1$ such that
\begin{equation}\label{eigp}
\left\{
\begin{aligned}
-\Delta u &= \lambda u,~{\rm in}~\OO,\\
u &= 0,~{\rm on}~\partial\OO.
\end{aligned}
\right.
\end{equation}

For the simplicity of notations, let $\bo{\Sigma}$ be the Sobolev space $\bo{H}(\ddiv,\OO)$ and $V$ be $L^2(\OO)$. Now we introduce $\bs=-\nabla u$ so that $\ddiv\bs = \lambda u$, and then rewrite the strong problem \eqref{eigp} into the weak mixed problem by integral by parts: find $(\lambda,\bs,u)\in\R\times\bo{\Sigma}\times V$ such that $\forall~(\bt,v)\in\bo{\Sigma}\times V$,
\begin{equation}\label{eigmp}
\left\{
\begin{aligned}
(\bs,\bt)-(\ddiv\bt,u)&= 0,\\
(\ddiv\bs,v)&= \lambda(u,v).
\end{aligned}
\right.
\end{equation}

There exists a series of eigenpairs $(\lambda_k,\bs_k,u_k)$ of \eqref{eigmp} such that
$$0<\lambda_1\leq\lambda_2\leq\cdots\leq\lambda_k\nearrow\infty$$
and $(u_k,u_l) = \delta_{k,l}$, where $\delta_{k,l}$ is the Kronecker symbol \cite{babuska}. An easy observation to \eqref{eigmp} gives that $(\bs_k,\bs_l) = (\ddiv\bs_l,u_k) = \lambda_k(u_k,u_l)$, which indicates that $\{\bs_k\}$ form an orthogonal sequence with $L^2$ norm $\{\lambda_k^{1/2}\}$. Denote the eigenspace of $\lambda_k$ as $E_k$, $\bo{F}_k$, respectively. That is,
$$E_k = \bigoplus_{j\in\Lambda_k}\spn\{u_j\}~{\rm and}~\bo{F}_k = \bigoplus_{j\in\Lambda_k}\spn\{\bs_j\},$$
where $\Lambda_k$ is the index set defined as $\Lambda_k = \{j:\lambda_j = \lambda_k\}$.

The RRT element is the counterpart of the triangular Raviart-Thomas element on rectangular meshes. Let $\T_h$ be a rectangular mesh grided by node vectors $X=\{x_i\}_{i=0}^{n_1}$ and $Y = \{y_j\}_{j=0}^{n_2}$:
$$x_0 <x_1<\cdots<x_{n_1},~y_0 <y_1<\cdots<y_{n_2}.$$
Denote the size of the element $K_{i,j}$ grid by $x = x_{i-1}, x = x_i$ and $y = y_{j-1}, y = y_j$ as $h_{x_i}\times h_{y_j}$, where $h_{x_i} = x_i - x_{i-1}$ and $h_{y_j} = y_j - y_{j-1}$ for $i = 1,2,\ldots,n_1$ and $j = 1,2,\ldots,n_2$. Since we only consider the elements in $\overline{\Omega}$, we may rewrite $\T_h$ as
$$\T_h = \{K_{i,j}:K_{i,j}\subset\overline{\Omega}\}$$
when $\overline{\Omega}$ is not rectangular. Through the paper, quantities with subscript $\cdot_h$ is mesh-dependent. Especially, difference operators with subscript $\cdot_h$ is done piece by piece. We say $\T_h$ is $a$-regular if there exists a constant $a$ independent of $i$ and $j$ such that $a^{-1}h_{y_j}\leq h_{x_i}\leq a h_{y_j}$ for all $K_{i,j}\in \T_h$. When $\T_h$ is 1-regular, we say $\T_h$ is uniform. If all the elements in $\T_h$ are congruent to each other, we say $\T_h$ is quasi-uniform. Obviously a uniform mesh must be quasi-unform. The mesh size of $\T_h$ is $h:=\max\{h_{x_i},h_{y_j}\}$. The lowest-order RRT element uses $\bo{H}(\ddiv)$-conforming functions as its basis. The finite element space pair is denoted as $\bo{\Sigma}_h\times V_h$, where
$$\bo{\Sigma}_h:=\Big\{\bt_h\in\bo{\Sigma}:\bt|_K\in Q_{10}(K)\times Q_{01}(K),~\forall K\in \T_h\Big\}$$
with $Q_{mn} = \Sp\{x^iy^j,~0\leq i\leq m,~0\leq j\leq n\}$, and
$$V_h:=\{v_h\in V:v_h|_{K}\in Q_{00}(K),~\forall K\in \T_h\}.$$
The discrete mixed problem is to find $(\lambda_h,\bs_h,u_h)\in\R\times\bo{\Sigma}_h\times V_h$ such that $\forall~(\bt_h,v_h)\in\bo{\Sigma}_h\times V_h$,
\begin{equation}\label{deigmp}
\left\{
\begin{aligned}
(\bs_h,\bt_h)-(\ddiv\bt_h,u_h)&= 0,\\
(\ddiv\bs_h,v_h)&= \lambda_h(u_h,v_h).
\end{aligned}
\right.
\end{equation}
Denote $\bo{\Sigma}_{h0}$ the kernel space of the divergence operator in $\bo{\Sigma}_h$:
$$\bo{\Sigma}_{h0} = \{\bt_h\in\bo{\Sigma}_h:\ddiv\bt_h = 0\}.$$
The stability condition is standard \cite{rt,brezzi,boffi} that there exist constants $\alpha_0,\beta_0>0$ independent to $h$ such that
$$(\bt_h,\bt_h)\geq \alpha_0 \|\bt_h\|_{\ddiv}^2,~\forall \bt_h\in \bo{\Sigma}_{h0},$$
$$\sup_{\bt_h\in \bo{\Sigma}_h}(\ddiv \bt_h,v_h)\geq \beta_0\|\bt_h\|_{\ddiv}\|v_h\|_0,~\forall v_h\in V_h.$$

Using the fact that $\ddiv \bo{\Sigma}_h\subset V_h$ we have $\ddiv\bs_h = \lambda_h u_h$.
There exists a series of eigenpairs $(\lambda_{k,h},\bs_{k,h},u_{k,h})$ of \eqref{deigmp} such that
$$0<\lambda_{1,h}\leq\lambda_{2,h}\leq\cdots\leq\lambda_{k,h}\leq \lambda_{N,h}$$
and $(u_{k,h},u_{l,h}) = \delta_{k,l}$, where $N$ is the number of elements of $\T_h$. An observation to \eqref{deigmp} gives $(\bs_{k,h},\bs_{l,h}) = (\ddiv\bs_{l,h},u_{k,h}) = \lambda_{k,h}(u_{k,h},u_{l,h})$, which indicates that $\{\bs_{k,h}\}$ form an orthogonal sequence with $L^2$ norm $\{\lambda_{k,h}^{1/2}\}$, just the same as the continuous case. The discrete eigenspaces are denoted as
$$E_{k,h} = \bigoplus_{j\in\Lambda_k}\spn\{u_{j,h}\},~\bo{F}_{k,h} = \bigoplus_{j\in\Lambda_k}\spn\{\bs_{j,h}\},$$
respectively.

Given $u\in H^3(\Omega)$, the error estimate of \eqref{deigmp} is standard \cite{boffi,brezzi}:
\begin{eqnarray}
\lambda_h-\lambda = O(h^2),\label{erreig}\\
\|\bs_h-\bs\|_0 = O(h),\label{errsigma}\\
\|\ddiv(\bs_h-\bs)\|_0 = O(h),\\
\|u_h-u\|_0 = O(h).\label{erru}
\end{eqnarray}
The spectral theory \cite{babuska} gives that $0$ is the only accumulation point of $\lambda^{-1}$. Then there exists \enquote{gap} $\rho_k$ for each eigenvalue $\lambda_k$ such that
$$\frac{\lambda_k}{|\lambda_k-\lambda_{j,h}|}\leq \rho_k,~\forall j~{\rm s.t.}~\lambda_j\neq\lambda_k,$$
when $h$ is small enough.

Let $R$ and $S$ be two finite dimensional subspaces of a Hilbert space $H$ endowed with norm $\|\cdot\|_H$ and $\dim(R) = \dim(S)>1$. Let $P_R$ and $P_S$ be the orthogonal projection operators onto $R$ and $S$, respectively. The gap between the two spaces $R$ and $S$ is defined as
$$\delta(R,S) = \sup_{x\in R,\|x\|_H=1}\|x-P_Sx\|_H.$$
Another definition gives $\tilde{\delta}(R,S) = \|P_R-P_S\|$, where $\|\cdot\|$ is the operator norm induced by $\|\cdot\|_H$. In the settings that $R$ and $S$ are finite dimensional Hilbert spaces, we have $\delta(R,S) = \delta(R,S) = \tilde{\delta}(R,S)$ \cite[Lemma 221]{kato58}. Hence it is no need to distinguish them.

For multiple eigenvalues $\lambda_k$, we have \cite{boffi,brezzi}
\begin{eqnarray}
\delta(E_k,E_{k,h}) = O(h),\\
\delta(\bo{F}_k,\bo{F}_{k,h}) = O(h).
\end{eqnarray}

Let $\bs_I\in V_h$ be the interpolation of $\bs$, i.e., for each edge $e$ of $K_{i,j}$,
\begin{equation}\label{intp}
\int_{e}(\bs-\bs_I)\cdot\bo{n}\dd s = 0.
\end{equation}
The interpolation error is standard \cite{brezzi}:
\begin{equation}
\|\bs-\bs_I\|_0 = O(h).\label{errintp}
\end{equation}

By Green Formula we have
$$\int_{K_{i,j}}\ddiv(\bs-\bs_I)\dd x\dd y = \sum_{e\subset\partial K_{i,j}}\int_{e}(\bs-\bs_I)\cdot\bo{n}\dd s = 0.$$
Then for all $v_h\in V_h$, it is easy to see
\begin{equation}\label{erreqintp}
(\ddiv(\bs-\bs_I),v_h) = 0.
\end{equation}

\section{Superconvergence}
\label{sec:spc}
In this section, we shall prove the supercloseness property between the finite element eigenfunctions and the interpolated eigenfunctions, and derive superconvergence after postprocessing. The supercloseness property is also helpful in the expansion of the error of eigenvalues.

The Ritz projection of eigenfunction is denoted as $(R_h\bs,S_hu)\in\bo{\Sigma}_h\times V_h$, which solves the problem
\begin{equation}\label{deigRh}
\left\{
\begin{aligned}
(R_h\bs,\bt_h)-(\ddiv\bt_h,S_hu) &= 0,&~&\forall \bt_h\in\bo{\Sigma}_h,\\
(\ddiv R_h\bs,v_h) &= \lambda(u,v_h),&~&\forall v_h\in V_h.
\end{aligned}
\right.
\end{equation}
The error equation gives
\begin{equation}\label{erreqRh}
\left\{
\begin{aligned}
(\bs-R_h\bs,\bt_h)-(\ddiv\bt_h,u-S_hu) &= 0,&~&\forall \bt_h\in\bo{\Sigma}_h,\\
(\ddiv (\bs-R_h\bs),v_h) &= 0,&~&\forall v_h\in V_h.
\end{aligned}
\right.
\end{equation}

The following Lemmas \ref{lem:integralExpansion} and \ref{lem:Ritesupc} indicate superconvergence for the Laplace problem.
\begin{lemma}[\cite{lin06},Theorem 4.27]\label{lem:integralExpansion}
Suppose that $\bs\in \bo{H}^3(\OO)\cap\bo{\Sigma}$, then for all $\bt_h = (\tau_{h1},\tau_{h2})^T\in \bo{\Sigma}_h$,
\begin{multline}\label{explem0}
(\bs-\bs_I,\bt_h) = -\frac13\sum_{K_{i,j}\in \T_h}\Big(\big(\frac{h_{x_i}}{2}\big)^2\int_{K_{i,j}}-u_{xxx}\tau_{h1}\dd x\dd y \\
+ \big(\frac{h_{y_j}}{2}\big)^2\int_{K_{i,j}}-u_{yyy}\tau_{h2}\dd x\dd y\Big) + O(h^3).
\end{multline}
\end{lemma}

\begin{lemma}\label{lem:Ritesupc}
Suppose $u\in H^4(\OO)$. Let $(R_h\bs,S_hu)$ be the Ritz projection satisfying \eqref{deigRh}, the supercloseness property is valid:
\begin{eqnarray}
\|\bs_I-R_h\bs\|_0 = O(h^2),&\label{supcs0}\\
\|\ddiv(\bs_I-R_h\bs)\|_0 = 0,&\label{supcds0}\\
\|\Pi_0u-S_hu\|_0 = O(h^2),&\label{supcu0}
\end{eqnarray}
where $\bs_I$ is the interpolation of $\bs$ defined by \eqref{intp} and $\Pi_0u$ is the $L^2$ orthogonal projection to $V_h$.
\end{lemma}
\begin{proof}
Noting \eqref{erreqintp}, we rewrite the error equation \eqref{erreqRh} as following: for all $(\bt_h,v_h)\in\bo{\Sigma}_h\times V_h$,
\begin{equation}\label{erreqRh2}
\left\{
\begin{aligned}
(\bs_I-R_h\bs,\bt_h)-(\ddiv\bt_h,\Pi_0u-S_hu) &= (\bs_I-\bs,\bt_h),\\
(\ddiv(\bs_I-R_h\bs),v_h) &= 0,
\end{aligned}
\right.
\end{equation}
and \eqref{supcds0} is obvious since we can take $v_h=\ddiv(\bs_I-R_h\bs)$ in \eqref{erreqRh2}.

Taking $(\bt_h,v_h) = (\bs_I-R_h\bs,\Pi_0u-S_hu)$ in \eqref{erreqRh2}, we have
\begin{equation}\label{posc1}
\|\bs_I-R_h\bs\|_0^2 = (\bs_I-\bs,\bs_I-R_h\bs) + (\ddiv(\bs_I-R_h\bs),\Pi_0u-S_hu)
\end{equation}
as well as
\begin{equation}\label{posc3}
(\ddiv(\bs_I-R_h\bs),\Pi_0u-S_hu) = 0.
\end{equation}
The integral expansion \eqref{explem0} gives
\begin{equation}\label{posc4}
(\bs_I-\bs,\bs_I-R_h\bs) = O(h^2)\|\bs_I-R_h\bs\|_0.
\end{equation}
Then \eqref{supcs0} follows from \eqref{posc1}-\eqref{posc4}.

From the stability conditions, we have
\begin{equation}
\begin{aligned}
\|\Pi_0u-S_hu\|_0&\leq \frac{1}{\beta_0}\sup\limits_{\bt_h\in\bo{\Sigma}_h}\frac{(\ddiv\bt_h,\Pi_0u-S_hu)}{\|\bt_h\|_{\ddiv}}\\
&\leq\frac{1}{\beta_0}\Big(\sup\limits_{\bt_h\in\bo{\Sigma}_h}\frac{(\bs_I-R_h\bs,\bt_h)}{\|\bt_h\|_{\ddiv}} + \sup\limits_{\bt_h\in\bo{\Sigma}_h}\frac{(\bs_I-\bs,\bt_h)}{\|\bt_h\|_{\ddiv}} \Big)\\
&\leq\frac{1}{\beta_0}(\|\bs_I-R_h\bs\|_0 + O(h^2)),
\end{aligned}
\end{equation}
which guarantees \eqref{supcu0}.
\end{proof}

\begin{theorem}[Supercloseness]\label{thmsupc}
Let $(\bs_h,u_h)$ be the eigenfunction of \eqref{deigmp}. Assume that the corresponding eigenspaces of \eqref{eigmp} $\bo{F}$ and $E$ satisfy $\bo{F}\times E\subset \bo{H}^4(\Omega)\times H^3(\Omega)$. Then there exists $(\bs,u)\in \bo{F}\times E$ such that
\begin{eqnarray}
\|\bs_I-\bs_h\|_0 = O(h^2),&\label{supcs}\\
\|\ddiv(\bs_I-\bs_h)\|_0 = O(h^2),&\label{supcds}\\
\|\Pi_0u-u_h\|_0 = O(h^2),&\label{supcu}
\end{eqnarray}
where $\bs_I$ is the interpolation of $\bs$ defined by \eqref{intp} and $\Pi_0u$ is the $L^2$ orthogonal projection to $V_h$.
\end{theorem}
\begin{proof}
Let us consider the case that $\lambda_k$ is simple, which means that $\lambda_j\neq\lambda_k$ as long as $j\neq k$. We first prove
\begin{equation}\label{posce0}
\|S_hu_k-u_{k,h}\|_0=O(h^2).
\end{equation}
Since $\{u_{j,h}\}$ forms orthonormal basis of $V_h$, we decompose $S_hu_k$ as
$$S_h u_k = w_{k,h} + \sum_{j\neq k}(S_hu,u_{j,h})u_{j,h},$$
where $w_{k,h} = (S_hu_k,u_{k,h})u_{k,h}$. Then
\begin{equation}\label{posce1}
\|S_hu_k-u_{k,h}\|_0^2 = \|w_{k,h}-u_{k,h}\|_0^2 + \sum_{j\neq k}(S_hu_j,u_{j,h})^2.
\end{equation}
Note that $u_{k,h}$ is unique up to its sign, we choose the one such that $(S_hu_k,u_{k,h})>0$. Hence
$$\|w_{k,h}-u_{k,h}\|_0 = |1-(S_hu_k,u_{k,h})|\leq |1-(\Pi_0u_k,u_{k,h})| + \|\Pi_0u_k - S_hu_k\|_0,$$
we recall that $|1-(\Pi_0u_k,u_{k,h})| = |1-(u_k,u_{k,h})| = \frac12\|u_k-u_{k,h}\|_0^2=O(h^2)$ and \eqref{supcu0}, and then we prove
\begin{equation}\label{posce2}
\|w_{k,h}-u_{k,h}\|_0 = O(h^2).
\end{equation}

For all $j\neq k$, we have
$$\lambda_{j,h}(S_hu_k,u_{j,h}) = (S_hu_k,\ddiv\bs_{j,h}) = (R_h\bs_k,\bs_{j,h}) = (\ddiv R_h\bs_k,u_{j,h}) = \lambda_k(\Pi_0u_k,u_{j,h}),$$
and then
\begin{equation*}
(S_hu_k,u_{j,h}) = \frac{\lambda_k}{\lambda_{j,h} - \lambda_k}(\Pi_0u_k - S_hu_k,u_{j,h}).
\end{equation*}
Recalling $\frac{\lambda_k}{\lambda_{j,h} - \lambda_k}\leq\rho_k$ and \eqref{supcu0}, we have
\begin{equation}\label{posce3}
\sum_{j\neq k}(S_hu,u_{j,h})^2\leq\rho_k^2\|\Pi_0u_k-S_hu_k\|_0^2 = O(h^4).
\end{equation}
Then \eqref{posce0} follows from \eqref{posce1}-\eqref{posce3}. Combining \eqref{posce0} and \eqref{supcu0}, by applying the triangle inequality we get \eqref{supcu}.

The error equation of \eqref{deigmp} gives that for all $(\bt_h,v_h)\in\bo{\Sigma}_h\times V_h$,
\begin{equation}\label{posce4}
\left\{
\begin{aligned}
(\bs_I-\bs_h,\bt_h)-(\ddiv\bt_h,\Pi_0u-u_h) &= (\bs_I-\bs,\bt_h),\\
(\ddiv(\bs_I-\bs_h),v_h) &= \lambda(u,v_h) - \lambda_h(u_h,v_h),
\end{aligned}
\right.
\end{equation}
Taking $(\bt_h,v_h) = (\bs_I-\bs_h,\Pi_0u-u_h)$, we obtain
\begin{equation}
\begin{aligned}
\|\bs_I-\bs_h\|_0^2 =&~(\bs_I-\bs,\bs_I-\bs_h) + (\ddiv(\bs_I-\bs_h),\Pi_0u-u_h)\\
=&~(\bs_I-\bs,\bs_I-\bs_h) + \lambda(u,\Pi_0u-u_h) - \lambda_h(u_h,\Pi_0u-u_h)\\
=&~(\bs_I-\bs,\bs_I-\bs_h) + (\lambda-\lambda_h)(\Pi_0u,\Pi_0u-u_h) + \lambda_h(\Pi_0u-u_h,\Pi_0u-u_h)\\
=:&~I_1+I_2+I_3.
\end{aligned}
\end{equation}
From \eqref{erreig}, \eqref{explem0} and \eqref{supcu}, we can get
\begin{equation}
\begin{aligned}
|I_1| &= O(h^2)\|\bs_I-\bs_h\|_0\leq \frac12(\|\bs_I-\bs_h\|_0^2 + O(h^4)),\\
|I_2| &\leq |\lambda-\lambda_h|\|\Pi_0u-u_h\|_0 = O(h^4),\\
|I_3| &= \lambda_h\|\Pi_0u-u_h\|_0^2 = O(h^4).
\end{aligned}
\end{equation}
Then the proof of \eqref{supcs} is completed.

Direct computation gives
$$\ddiv(\bs_I-\bs_h) = \lambda\Pi_0u - \lambda_hu_h = (\lambda-\lambda_h)\Pi_0u + \lambda_h(\Pi_0u-u_h)$$
and then from \eqref{erreig} and \eqref{supcu}, equation \eqref{supcds} holds.

Now we consider the case that $\lambda_k$ is multiple. The proof needs some mild modifications. In the rest of the proof, when we say \enquote{eigenfunction} we refer to $u$ rather $\bs$. Recall that $E_k$ is the eigenspace of $\lambda_k$ and $E_{k,h}$ the corresponding discrete eigensapce. Denote $P_{E_{k,h}}$ the orthogonal projection onto $E_{k,h}$. The gap between $E_h$ and $E_{k,h}$ is of rate $O(h)$ implies that $P_{E_{k,h}}$ is bijective between $E_k$ and $E_{k,h}$. Then for $u_{k,h}$ we can choose the $u_k\in E_k,~\|u_k\|_0=1$ such that $P_{E_{k,h}}u_k = c u_{k,h}$ with $c = \sqrt{1-O(h^2)} = 1-O(h^2)$. Recall the index set $\Lambda_k = \{j:\lambda_j = \lambda_k\}$, and for $j\in\Lambda_k$, we have
$$(S_hu_k,u_{j,h}) = \lambda_{j,h}^{-1}(R_h\bs_k,\bs_{j,h}) = \lambda_k\lambda_{j,h}^{-1}(u_k,u_{j,h}) = c\lambda_k\lambda_{j,h}^{-1}(u_{k,h},u_{j,h}) = c\lambda_k\lambda_{j,h}^{-1}\delta_{k,j}.$$
The decomposition of $S_hu_k$ becomes
$$S_hu_k = w_{k,h} + \sum_{j\notin \Lambda_k}(S_hu_k,u_{j,h})u_{j,h}$$
with $w_{k,h} = (S_hu_k,u_{k,h})u_{k,h}$ and $\|w_{k,h} - u_{k,h}\|_0 = |1-c\lambda_k\lambda_{k,h}^{-1}| = O(h^2)$. The rest of proof is the same line, except the summation condition $j\neq k$ needs to be substituted with $j\notin\Lambda_k$.
\end{proof}

For the superconvergence of eigenfunctions $\bs_h$ and $u_h$, we utilize postprocessing skills \cite{lin06}. Assume that $\T_h$ is the uniform refined mesh of $\T_{2h}$. Let $I_{2h}^1$ and $J_{2h}^1$ be the interpolation operators defined in \cite[Example 7.14]{lin06}, i.e., $I_{2h}^1$ and $J_{2h}^1$ preserve functions in $Q_{11}$ spaces, vector-valued or scalar-valued. Some properties of the operators $I_{2h}^1$ and $J_{2h}^1$ are valid:

\begin{lemma}\label{post}
For all $\bt\in\bo{H}^2(\OO)$, $v\in H^2(\OO)$ and $(\bt_h,v_h)\in\bo{\Sigma}_h\times V_h$,
\begin{eqnarray*}
&I_{2h}^1\bt = I_{2h}\bt_I,~J_{2h}^1v = J_{2h}^1\Pi_0v,\\
&\|\bt-I_{2h}^1\bt\|_l = O(h^{2-l}),~\|v-J_{2h}^1v\|_l = O(h^{2-l}),~l = 0,1,\\
&\|I_{2h}^1\bt_h\|_l\leq C\|\bt_h\|_l,~\|J_{2h}^1v_h\|_l\leq C\|v_h\|_l.
\end{eqnarray*}
\end{lemma}

\begin{theorem}[Superconvergence]
Suppose $(\bs,u)$ and $(\bs_h,u_h)$ are eigenfunctions of problems \eqref{eigmp} and \eqref{deigmp}, respectively. Suppose the conditions of Theorem \ref{thmsupc} and Lemma \ref{post} are valid, we have the superconvergence for eigenfunctions after postprocessing:
\begin{eqnarray*}
&\|I_{2h}^1\bs_h - \bs\|_l=O(h^{2-l}),&\\
&\|J_{2h}^1u_h - u\|_l = O(h^{2-l}),&l =0,1.
\end{eqnarray*}
\end{theorem}
\begin{proof}
The proof is nothing but direct use of Theorem \ref{thmsupc} and Lemma \ref{post}.
\begin{equation*}
\begin{aligned}
\|I_{2h}^1\bs_h - \bs\|_l &\leq \|I_{2h}^1\bs_h - I_{2h}^1\bs_I\|_l + \|I_{2h}^1\bs - \bs\|_l \\
&\leq C\|\bs_h - \bs_I\|_l + \|I_{2h}^1\bs - \bs\|_l= O(h^{2-l}),
\end{aligned}
\end{equation*}
\begin{equation*}
\begin{aligned}
\|J_{2h}^1u_h - u\|_l &\leq \|J_{2h}^1u_h - J_{2h}^1u_I\|_l + \|J_{2h}^1u - u\|_l\\
&\leq C\|u_h - u_I\|_l + \|J_{2h}^1u - u\|_l = O(h^{2-l}).
\end{aligned}
\end{equation*}
\end{proof}

\section{Error expansion for eigenvalues}
\label{sec:exp}

In this section, we shall introduce a crude error expansion \eqref{expansion} for eigenvalues of mixed elements, and then by applying basis integral methods in \cite{lin06} for our expansion, we get the error expansion \eqref{explem} for eigenvalues with optimal residual. After that, we propose a delicate expansion \eqref{result} for the error of eigenvalue by applying the superconvergence results of Theorem \ref{thmsupc}, and predict the convergence behavior of multiple eigenvalues on uniform mesh (Theorem \ref{thm:multiple}) for the problem on square domain and the convergence of extrapolation for eigenvalues (Theorem \ref{thm:extra}). As byproducts, we prove the upper bound property for discrete eigenvalues on rectangular domains (Theorem \ref{thm:upperbound}), and provide a lower bound for $\lambda_h-\lambda$ (Theorem \ref{thm:errlowerbd}).

Recall we have defined the Ritz projection $(R_h\bs,S_hu)$ in Section \ref{sec:spc}. With the help of $R_h\bs$ and $S_hu$, we can derive a crude expansion identity for $\lambda_h-\lambda$:
\begin{theorem}\label{thm:crudeExpansion}
Suppose $(\lambda_h,\bs_h,u_h)$ is an eigenpair of the discrete problem \eqref{deigmp} and $(\lambda,\bs,u)$ is the corresponding eigenpair of \eqref{eigmp} with $u\in H^3(\Omega)$. Then we have the error expansion identity \eqref{expansion} for $\lambda_h$:
\begin{equation}\label{expansion}
\lambda_h-\lambda = (\bs-\bs_I,\bs_h) + O(h^4),
\end{equation}
where $\bs_I$ is the interpolation of $\bs$ defined by \eqref{intp}.
\end{theorem}
\begin{proof}
From \eqref{eigmp} and \eqref{deigmp} we have
$$(\bs,\bs_h) = (\ddiv \bs_h,u) = \lambda_h(u,u_h).$$
Similarly, from \eqref{deigmp} and \eqref{deigRh}, we have
$$(R_h\bs,\bs_h) = (\ddiv R_h\bs,u_h) = \lambda(u,u_h).$$
Put $\hat{\bs}_h = \frac{\bs_h}{(u,u_h)}$ and $\hat{u}_h = \frac{u_h}{(u,u_h)}$. From the definition of $\hat{\bs}_h$, it is easy to see $(\bs,\hat{\bs}_h) = \lambda_h$ and $(R_h\bs,\hat{\bs}_h) = \lambda$.

Then
\begin{equation}\label{exp1}
\begin{aligned}
\lambda_h-\lambda &= (\bs-R_h\bs,\hat{\bs}_h)\\
&= (\bs-R_h\bs,\bs_h)(1+\frac{1-(u,u_h)}{(u,u_h)})\\
&= (\bs-R_h\bs,\bs_h)(1+O(h^2)),
\end{aligned}
\end{equation}
where we have used \eqref{erru} and
$$1-(u,u_h)=\frac{1}{2}((u,u) + (u_h,u_h) - 2(u,u_h)) = \frac{1}{2}\|u-u_h\|_0^2.$$
Note that
\begin{equation}\label{expsp}
(\bs-R_h\bs,\bs_h) = (\bs-\bs_I,\bs_h) + (\bs_I-R_h\bs,\bs_h),
\end{equation}
and
\begin{equation}\label{vanish}
\begin{aligned}
(\bs_I-R_h\bs,\bs_h) &= (\ddiv(\bs_I-R_h\bs),u_h)\\
&= (\ddiv(\bs_I-\bs),u_h) + (\ddiv(\bs-R_h\bs),u_h)\\
&= 0,
\end{aligned}
\end{equation}
where we have used $u_h$ is piecewise constant and then \eqref{erreqintp} and \eqref{erreqRh} are valid. Combining \eqref{erreig} and \eqref{exp1}-\eqref{vanish} we get the desired result \eqref{expansion}.
\end{proof}

In order to derive a further expansion for $\lambda_h-\lambda$, it is necessary to derive an expansion for $(\bs-\bs_I,\bs_h)$. We make use of the tools introduced in \cite{lin06}. Lemma \ref{lem:integralExpansion} provides an available expansion of the integral term. However, the convergence rate of residual in \eqref{explem0} is suboptimal. Hence a bit of more work needs to do. We use the framework in \cite{lin06} to get the optimal expansion. Let $K$ be the considered bounded domain. Suppose that $B(\cdot,\cdot):H^{k+1}(K)\times S(K)\rightarrow \mathbb{R}$ is a bounded bilinear form with $B(u,v)\leq c\|u\|_{k+1,K}|v|_{j,K}$ for some integer $j$, where $S(K)$ is a finite dimensional polynomial space on $K$. Denote $P_k(K)$ the polynomial function space of degree not greater than $k$. Assume that
\begin{equation}\label{asmp1}
B(u,v) = 0,~\forall (u,v)\in P_k(K)\times S(K),
\end{equation}
and for each multi-index $\alpha$ with $|\alpha| = k+1$, we can find a differential operator $D^{\gamma}$ of order $\gamma\geq j$ such that
\begin{equation}\label{asmp2}
B(x^{\alpha},v) = \frac{1}{|K|}\int_KD^{\gamma}v\dd x\dd y,~\forall v\in S(K).
\end{equation}
Now we introduce a expansion lemma in \cite{lin06}, which can be straightforwardly proved from a generalized Bramble-Hilbert lemma.
\begin{lemma}[\cite{lin06},Lemma 2.6]\label{lem:expansionMethod}
If the bilinear form $B(u,v)$ satisfies \eqref{asmp1} and \eqref{asmp2}, then
$$B(u,v)=\sum_{|\beta|=k+1}\frac{1}{\beta!|K|}\int_KD^{\beta}uD^{\gamma}v\dd x\dd y + H(u,v),$$
where $H(u,v)$ satisfies
$$|H(u,v)| \leq c|u|_{k+2,K}|v|_{j,K},$$
with $c$ independent of $u$ and $v$.
\end{lemma}
By applying Lemma \ref{lem:expansionMethod}, we can prove the optimal convergence rate of the residual in Lemma \ref{lem:integralExpansion}.
\begin{theorem}\label{thm:integralExpansion}
Suppose that $\bs\in\bo{H}^4(\OO)$, then for all $\bt_h = (\tau_{h1},\tau_{h2})^T\in \bo{\Sigma}_h$,
\begin{equation}\label{explem}
(\bs-\bs_I,\bt_h) = \frac{1}{12}\sum_{K_{i,j}\in \T_h}\Big(h_{x_i}^2\int_{K_{i,j}}u_{xxx}\tau_{h1}\dd x\dd y + h_{y_j}^2\int_{K_{i,j}}u_{yyy}\tau_{h2}\dd x\dd y\Big) + O(h^4).
\end{equation}
\end{theorem}
\begin{proof}
Let $K$ be the reference element $[-1,1]^2$, on which $S(K) = \spn\{1,x\}$. Let $(\sigma,\tau)^T$ be smooth on $K$. Let $(\sigma_I,\tau_I)^T$ be the interpolation defined by \eqref{intp}. Then one can verify $(\sigma_I,0)^T$ is the interpolation of $(\sigma,0)^T$ and so is $(0,\tau_I)^T$. We consider the bilinear form
$$B(\sigma,v) = (\sigma-\sigma_I,v) + \frac{1}{3}\int_K\sigma_{xx}v\dd x\dd y.$$
For $$\sigma = 1,x,y,xy,x^2,y^2,x^2y,xy^2,x^3,y^3,$$
direct computation gives $$\sigma_I = 1,x,0,0,1,\frac13,0,\frac x3,x,0,$$
respectively. Then we need to verify for each $\sigma\in P_2(K)$, $B(\sigma,[1,x]) = [0,0]$.

$B(1,[1,x]) = B(x,[1,x]) = [0,0]$ is obvious.

$B(y,[1,x]) = \int_Ky[1,x]\dd x\dd y = [0,0]$,

$B(xy,[1,x]) = \int_Kxy[1,x]\dd x\dd y = [0,0]$,

$B(x^2,[1,x]) = \int_K(x^2-1 + \frac23)[1,x]\dd x\dd y = [0,0]$,

$B(y^2,[1,x]) = \int_K(y^2-\frac13)[1,x]\dd x\dd y = [0,0]$. Yet we have yielded $B(\sigma,v)$ vanishes for all $(\sigma,v)\in P_2(K)\times S(K)$.

Further computation gives $B(x^2y,[1,x]) = \int_K(x^2y-2y)[1,x]\dd x\dd y = [0,0]$,

$B(xy^2,[1,x]) = \int_K(xy^2-\frac x3)[1,x]\dd x\dd y = [0,0]$,

$B(x^3,[1,x]) = \int_K(x^3-x+2x)[1,x]\dd x\dd y = [0,\frac{32}{15}]$,

$B(y^3,[1,x]) = \int_Ky^3[1,x]\dd x\dd y = [0,0]$.

The only term we need to deal with is $B(x^3,[1,x])$. Now for $\alpha = (3,0)$, taking $D^{\gamma}v = \frac{32}{15}v_x$, we have $B(\bo{x}^{\alpha},[1,x]) = \frac{1}{|K|}\int_KD^{\gamma}[1,x]\dd x\dd y$.

Then from Lemma \ref{lem:expansionMethod}, we have
$$B(\sigma,v)=\frac{4}{45}\int_K\sigma_{xxx}v_x\dd x\dd y + H(\sigma,v),$$
where $|H(\sigma,v)|\leq c|\sigma|_4\|v\|_0$.

By applying standard scaling argument \cite{lin06}, we have
$$(\sigma-\sigma_I,v) = -\frac{1}{12}h_{x_i}^2\sum_{K_{i,j}\in \T_h}\int_{K_{i,j}}\sigma_{xx}v\dd x\dd y + \frac{1}{180}h_{x_i}^4\sum_{K_{i,j}\in \T_h}\int_{K_{i,j}}\sigma_{xxx}v_x\dd x\dd y + O(h^4),$$
and the proof of
$$(\tau-\tau_I,v) = -\frac{1}{12}h_{y_j}^2\sum_{K_{i,j}\in \T_h}\int_{K_{i,j}}\tau_{yy}v\dd x\dd y + \frac{1}{180}h_{y_j}^4\sum_{K_{i,j}\in \T_h}\int_{K_{i,j}}\tau_{yyy}v_y\dd x\dd y + O(h^4)$$
is the same line.

Summing up the above two equations and noting $\bs = (-u_x,-u_y)^T$, we get \eqref{explem}.
\end{proof}

Now we can derive our main result of the error expansion of eigenvalues.
\begin{theorem}\label{thm:delicateExpansion}
Suppose $(\lambda_h,\bs_h,u_h)$ is an eigenpair of the discrete problem \eqref{deigmp} and $(\lambda,\bs,u)$ is the corresponding eigenpair of \eqref{eigmp} with $u\in H^5(\Omega)$. Then we have the error expansion identity \eqref{result} for $\lambda_h$:
\begin{equation}\label{result}
\lambda_h-\lambda = -\frac{1}{12}\sum_{K_{i,j}\in \T_h}\Big(h_{x_i}^2\int_{K_{i,j}}u_{xxx}u_x\dd x\dd y + h_{y_j}^2\int_{K_{i,j}}u_{yyy}u_y\dd x\dd y\Big) + O(h^4).
\end{equation}
\end{theorem}
\begin{proof}
We want to prove
\begin{equation}\label{dominant}
(\bs-\bs_I,\bs_h) = -\frac{1}{12}\sum_{K_{i,j}\in \T_h}\Big(h_{x_i}^2\int_{K_{i,j}}u_{xxx}u_x\dd x\dd y + h_{y_j}^2\int_{K_{i,j}}u_{yyy}u_y\dd x\dd y\Big) + O(h^4).
\end{equation}

Denote $\bs$ as $(\sigma,\tau)^T$, $\bs_h$ as $(\sigma_h,\tau_h)^T$, and the interpolation of $(u_{xxx},u_{yyy})^T$ as $(\xi,\eta)^T$. Consider the difference of \eqref{dominant} and \eqref{explem}. If we ignore the difference between $h_{x_i}$ and $h_{y_j}$, we have
\begin{equation}
\begin{aligned}
\int_{\OO}u_{xxx}(\sigma - \sigma_h)\dd x\dd y =& \int_{\OO}u_{xxx}(\sigma - \sigma_I)\dd x\dd y + \int_{\OO}u_{xxx}(\sigma_I - \sigma_h)\dd x\dd y\\
=& \int_{\OO}(u_{xxx}-\xi)(\sigma - \sigma_I)\dd x\dd y + \int_{\OO}\xi(\sigma - \sigma_I)\dd x\dd y\\
&+ \int_{\OO}u_{xxx}(\sigma_I - \sigma_h)\dd x\dd y\\
=:&~I_1+I_2+I_3.
\end{aligned}
\end{equation}

Note that \eqref{errintp}, the Cauchy-Schwarz inequality gives $I_1 = O(h^2)$. $I_2=O(h^2)$ follows from Theorem \ref{thm:integralExpansion}. The superconvergence \eqref{supcs} gives $I_3 = O(h^2)$. Thus we have
\begin{equation}\label{sigma}
h^2\int_{\OO}u_{xxx}(\sigma - \sigma_h)\dd x\dd y = O(h^4).
\end{equation}
The proof of
\begin{equation}\label{tau}
h^2\int_{\OO}u_{yyy}(\tau - \tau_h)\dd x\dd y = O(h^4)
\end{equation}
is the same line.

Combining \eqref{explem}, \eqref{sigma}, and \eqref{tau} we obtain \eqref{dominant}, from which \eqref{result} follows.
\end{proof}

Now let us consider the case where $\Omega$ is a rectangle, on which the eigenfunctions are tensor product of sine waves. We can rewrite \eqref{result} as
\begin{equation}\label{result2}
\lambda_h-\lambda = \frac{1}{12}\sum_{K_{i,j}\in \T_h}\Big(h_{x_i}^2\int_{K_{i,j}}u_{xx}^2\dd x\dd y + h_{y_j}^2\int_{K_{i,j}}u_{yy}^2\dd x\dd y\Big) + O(h^4)
\end{equation}
after integral by part.

Noting that the coefficient of the $h^2$ term in \eqref{result2} is always positive, we have Theorem \ref{thm:upperbound}.
\begin{theorem}\label{thm:upperbound}
For the lowest-order RRT element on rectangular domains, the computed eigenvalues provides asymptotic upper bounds for exact ones on general rectangular meshes, i.e. $\lambda_h\geq\lambda$ for $h$ small enough.
\end{theorem}

When $\Omega$ is a rectangle, we can also propose a lower bound for error of eigenvalues. 
\begin{theorem}\label{thm:errlowerbd}
Assume $\Omega$ is a rectangle and $\T_h$ is $a$-regular, then
\begin{equation}
\lambda_h-\lambda \geq C\lambda^2h^2 + O(h^4)
\end{equation}
if the corresponding eigenfunction $u\in H^5(\Omega)$, where $C=\frac{h^2}{12a^2}$ independent of $h$ and $\lambda$.
\end{theorem}
\begin{proof}
From \eqref{result2} we have
\begin{multline*}
\lambda_h-\lambda\geq\frac{h^2}{12a^2}\int_{\Omega}u_{xx}^2+u_{yy}^2\dd x\dd y + O(h^4)\geq\frac{h^2}{24a^2}\int_{\Omega}(u_{xx}+u_{yy})^2\dd x\dd y + O(h^4)\\
= \frac{h^2}{24a^2} \|\Delta u\|_0^2 + O(h^4) = C\lambda^2h^2 + O(h^4).
\end{multline*}
\end{proof}

For simple eigenvalues, it is sufficient to provide a explicit expansion by using \eqref{result2}. However, things go difficult if $\lambda$ is multiple: we cannot determine which function in the eigenspace of problem \eqref{eigmp} is the corresponding eigenfunction to the discrete one. We shall be careful to cope with things in this case. For the settings of Theorem \ref{thm:multiple} and Remark \ref{rmk:multiple}, we put $\OO = [0,\pi]^2$.

Put $u_{m,n}(x,y) = \frac{2}{\pi}\sin(mx)\sin(ny)$ for some $m,n\in\mathbb{N}^*$. For the considered domain $\OO = [0,\pi]^2$, theoretical results give that the exact eigenvalues are in the form $m^2+n^2$ and the corresponding eigenfunctions lie on the space spanned by $u_{m_*,n_*}(x,y)$ and $u_{n_*,m_*}(x,y)$ for all possible positive integral decomposition $m_*^2+n_*^2 = m^2+n^2$. Note that if for a simple eigenvalue $m^2+n^2$, $m$ is equal to $n$, but the converse does not hold, since we can easily find a counterexample $m = n =5$ with $m_*=1, n_*=7$. Suppose we have all decomposition pairs $\{(m_i,n_i)\}_{i=1}^p$ with $m_i^2+n_i^2 = \lambda_k$, where each pair $(m_i,n_i)$ is called a \enquote{frequency} and we do not distinguish the order of $m_i$ and $n_i$.

If $\T_h$ is uniform, i.e. $h_{x_i} \equiv h_x = h = h_y \equiv h_{y_j}$, \eqref{result2} reads 
\begin{equation}\label{corexp1}
\lambda_h-\lambda = \frac{h^2}{12}\int_{\OO}u_{xx}^2+u_{yy}^2\dd x\dd y + O(h^4).
\end{equation}
Furthermore, a delicate expansion \eqref{corexp2} is valid for both simple and multiple eigenvalues. However, the proof of \eqref{corexp2} is quite cumbersome, so we put it in the appendix.

\begin{theorem}\label{thm:multiple}
If $\T_h$ is uniform, then for $\lambda_k = m^2+n^2$, there exists an eigenvalue $\lambda_{h}^{m,n}$ of problem \eqref{deigmp} such that
\begin{equation}\label{corexp2}
\lambda_{h}^{m,n}-\lambda_k = \frac{(m^4+n^4)h^2}{12} + O(h^4),
\end{equation}
\end{theorem}

\begin{remark}\label{rmk:multiple}
For the problem on square domain, Theorem \ref{thm:multiple} implies that multiple eigenvalues of the same \enquote{frequency} share the same dominant term in error expansions when $\T_h$ is uniform. We will see this phenomenon in Section \ref{sec:num}, for example, $\lambda_{2,h}\approx\lambda_{3,h}$ and $\lambda_{5,h}\approx\lambda_{6,h}$ even in machine error. Moreover, the eigenfunctions tend to distinguish from each other by their \enquote{frequencies}: for each $(m,n)$ pair there exists an eigenspace in $\bo{\Sigma}_h$ close to $\spn\{u_{m,n},u_{n,m}\}$ with gap $O(h)$.
\end{remark}

Theorems \ref{thm:integralExpansion} and \ref{thm:delicateExpansion} indicate that we can improve the accuracy of $\lambda_h$ by extrapolation. Let $\tilde\lambda_h = \frac43\lambda_{h/2}-\frac13\lambda_h$, where $\lambda_{h/2}$ is the corresponding eigenvalue computed on the uniform refined mesh $\T_{h/2}$ of $\T_h$.
\begin{theorem}\label{thm:extra}
If $\lambda$ is simple with corresponding eigenfunction $u\in H^5(\Omega)$, or $\Omega$ is square with uniform partition $\T_h$, then
\begin{equation}\label{extra}
|\tilde\lambda_h-\lambda| = O(h^4).
\end{equation}
\end{theorem}
\begin{remark}
Though we only prove the extrapolation \eqref{extra} is valid when $\lambda$ is simple or $\T_h$ is uniform, our numerical experiments show \eqref{extra} still holds without the above restrictions.
\end{remark}

\section{Equivalence between the Projected ${\rm E}Q_1^{\rm rot}$ and RRT Elements}
\label{sec:equ}
In this section, we shall prove the equivalence between the projected ${\rm E}Q_1^{\rm rot}$ and RRT elements for the Laplace problem.

Let $\llbracket \cdot \rrbracket_e$ be the jump across $e\in \E_h$ and $\{\cdot\}_e$ be the average on $e$. The finite element space of the E$Q_1^{\rm rot}$ element is
$$V_h^{\rm EQ} := \{v_h\in L^2(\Omega):v_h|_K\in V_h^{\rm EQ}(K),\forall K\in\T_h, \llbracket v_h\rrbracket_e=0, \forall e\in \E_h\},$$
where $V_h^{\rm EQ}(K):=\spn\{1,x,y,x^2,y^2\}$ on $K\in\T_h$. The basis of $V_h^{\rm EQ}(K)$ can be chosen as $v_h^{e_i}$ and $v_h^K$ such that $\int_{e_j}v_h^{e_i}\dd e = \delta_{i,j}$, $\int_K v_h^{e_i}\dd\bo{x} = 0$, $\int_{e_i}v_h^K\dd e = 0$ and $\int_K v_h^K\dd\bo{x} = 1$.

For the Laplace problem, the projected ${\rm E}Q_1^{\rm rot}$ element seeks $u_h^{\rm PEQ}\in V_h^{\rm EQ}$ such that
\begin{equation}\label{eqnPEQ}
(\nabla_h u_h^{\rm PEQ},\nabla_h v_h) = (\Pi_0f,v_h),~\forall v_h\in V_h^{\rm EQ}
\end{equation}
with $\Pi_0$ the $L^2$ projector onto $V_h$. The resolvent operator $T_h^{\rm PEQ}:V\rightarrow V_h$ is defined by $T_h^{\rm PEQ}f=\Pi_0u_h^{\rm PEQ}$ for any $f\in V$, where $u_h^{\rm PEQ}$ is the solution of \eqref{eqnPEQ}.

When solving the Laplace problem with RRT element, we seek $(\bo{\sigma}_h^{\rm RRT},u_h^{\rm RRT})\in \bo{\Sigma}_h\times V_h$ such that
\begin{equation}\label{eqnRRT}
\left\{
\begin{aligned}
(\bo{\sigma}_h^{\rm RRT},\bo{\tau}_h) - (\ddiv\bo{\tau}_h,u_h^{\rm RRT}) &= 0,~\forall \bo{\tau}_h\in\bo{\Sigma}_h,\\
(\ddiv\bo{\sigma}_h^{\rm RRT},v_h) &= (f,v_h),~\forall v_h\in V_h.
\end{aligned}
\right.
\end{equation}
For the eigenvalue problem, we care about $u_h^{\rm RRT}$ rather $\bo{\sigma}_h^{\rm RRT}$ here, so we define the resolvent operator $T_h^{\rm RRT}:V\rightarrow V_h$ as $T_h^{\rm RRT}f = u_h^{\rm RRT}$ for any $f\in V$, where $u_h^{\rm RRT}$ is computed by \eqref{eqnRRT}.

\begin{theorem}
Let $u_h^{\rm PEQ}$ and $(\bo{\sigma}_h^{\rm RRT},u_h^{\rm RRT})$ be the solution of \eqref{eqnPEQ} and \eqref{eqnRRT}, respectively. Then we have
\begin{equation}
\bo{\sigma}_h^{\rm RRT} = -\nabla_h u_h^{\rm PEQ},~u_h^{\rm RRT} = \Pi_0u_h^{\rm PEQ}.
\end{equation}
\end{theorem}
\begin{proof}
We prove it in two steps.

First, we show $\nabla_h u_h^{\rm PEQ}\subset\bo{\Sigma}_h$. It is easy to see $\nabla V_h^{\rm EQ}(K)=Q_{10}(K)\times Q_{01}(K)$. Thus we only need to show $\nabla_h u_h^{\rm PEQ}\in\bo{\Sigma}$. For all $v_h\in V_h^{\rm EQ}$, since $\ddiv\bo{\sigma}_h^{\rm RRT}\in V_h$, integral by part yields
\begin{equation}
\begin{aligned}
(\bo{\sigma}_h^{\rm RRT},\nabla_hv_h) &= -(\ddiv\bo{\sigma}_h^{\rm RRT}, v_h) + \sum_{e\in\E_h}\int_e\bo{\sigma}_h^{\rm RRT}\cdot\bo{n}_e\llbracket v_h\rrbracket_e\dd e\\
&= -(\ddiv\bo{\sigma}_h^{\rm RRT}, \Pi_0v_h) + \sum_{e\in\E_h}\int_e\bo{\sigma}_h^{\rm RRT}\cdot\bo{n}_e\llbracket v_h\rrbracket_e\dd e
\end{aligned}
\end{equation}
Since $\bo{\sigma}_h^{\rm RRT}\cdot\bo{n}_e$ is constant on $e$ and $\int_e\llbracket v_h\rrbracket_e\dd e=0$, we have
$$\sum_{e\in\E_h}\int_e\bo{\sigma}_h^{\rm RRT}\cdot\bo{n}_e\llbracket v_h\rrbracket_e\dd e=0,$$
and then
\begin{equation}
(\bo{\sigma}_h^{\rm RRT},\nabla_hv_h) = -(\ddiv\bo{\sigma}_h^{\rm RRT}, \Pi_0v_h) = -(f,\Pi_0v_h) = -(\nabla_hu_h^{\rm PEQ},\nabla_hv_h).
\end{equation}
For any $e\in\E_h^i$, we choose $v_h^e\in V_h^{\rm EQ}$ satisfying both a) and b):
\begin{itemize}
\item[a)] $\int_ev_h^e\dd e = 1$ and $\int_{e'}v_h^e\dd e = 0$ for any $e'$ other than $e$,
\item[b)] $\int_Kv_h^e \dd\bo{x}= 0$ for any $K\in\T_h$.
\end{itemize}
Since $\ddiv(\bo{\sigma}_h^{\rm RRT} - \nabla_hu_h^{\rm PEQ})$ is piecewise constant, we have
\begin{equation}
\begin{aligned}
0 = (\bo{\sigma}_h^{\rm RRT} + \nabla_hu_h^{\rm PEQ},\nabla_hv_h^e) =& \int_e\llbracket\bo{\sigma}_h^{\rm RRT} + \nabla_hu_h^{\rm PEQ}\rrbracket_e\cdot\bo{n}_e\{v_h^e\}_e\dd e\\
&+ \int_e\{\bo{\sigma}_h^{\rm RRT} + \nabla_hu_h^{\rm PEQ}\}_e\cdot\bo{n}_e\llbracket v_h^e\rrbracket_e\dd e.
\end{aligned}
\end{equation}
Note that $\llbracket\bo{\sigma}_h^{\rm RRT} + \nabla_hu_h^{\rm PEQ}\rrbracket_e\cdot \bo{n}_e$ and $\{\bo{\sigma}_h^{\rm RRT} + \nabla_hu_h^{\rm PEQ}\}_e\cdot \bo{n}_e$ are constant on $e$, and then we can yield 
\begin{equation}
\llbracket\bo{\sigma}_h^{\rm RRT} + \nabla_hu_h^{\rm PEQ}\rrbracket_e\cdot \bo{n}_e=0.
\end{equation}

The second step is to show $(-\nabla_h u_h^{\rm PEQ},\Pi_0u_h^{\rm PEQ})$ satisfies \eqref{eqnRRT}. For any $\bo{\tau}_h\in\bo{\Sigma}_h$, we have $\bo{\tau}_h\cdot\bo{n}_e$ is constant on any $e\in\E_h$ and $\ddiv \bo{\tau}_h$ is piecewise constant, and then we get
\begin{equation}
\begin{aligned}
(-\nabla_h u_h^{\rm PEQ},\bo{\tau}) &= (\ddiv \bo{\tau}_h, u_h^{\rm PEQ}) - \sum_{e\in\E_h}\int_e\bo{\tau}\cdot\bo{n}_e\llbracket u_h^{\rm PEQ}\rrbracket_e\\
&= (\ddiv \bo{\tau}_h, u_h^{\rm PEQ}) = (\ddiv \bo{\tau}_h, \Pi_0u_h^{\rm PEQ}),
\end{aligned}
\end{equation}
which is the first equation of \eqref{eqnRRT}. For any $K\in\T_h$, we choose $v_h^K\in V_h^{\rm EQ}$ satisfying both c) and d):
\begin{itemize}
 \item[c)] $\int_Kv_h^K\dd\bo{x} = 1$ and $\int_{K'}v_h^K\dd\bo{x} = 0$ for any $K'\in\T_h$ other than $K$,
 \item[d)] $\int_ev_h^K\dd e = 0$ for any $e\in\E_h$.
\end{itemize}
Then, recall $\nabla_h u_h^{\rm PEQ}\cdot\bo{n}_e$ is a constant on $e$, and then we have
\begin{equation}
\begin{aligned}
(\ddiv(-\nabla_h u_h^{\rm PEQ}),v_h^K) &= (\nabla_h u_h^{\rm PEQ},\nabla_h v_h^K) - \sum_{e\in\partial K}\nabla_h u_h^{\rm PEQ}\cdot\bo{n}_e v_h^K\\
&= (\Pi_0f, v_h^K),
\end{aligned}
\end{equation}
which implies $\ddiv(-\nabla_h u_h^{\rm PEQ}) = \Pi_0f$ on $K$. Then the second equality of \eqref{eqnRRT} is verified and the proof is completed.
\end{proof}

Now let us consider the equivalence between the projected ${\rm E}Q_1^{\rm rot}$ and RRT Elements for the Laplace eigenvalue problem. The projected ${\rm E}Q_h^{\rm rot}$ element seeks $(\lambda_h^{\rm PEQ},u_h^{\rm PEQ})\in \mathbb{R}\times V_h^{\rm EQ}$ with $\|\Pi_0u_h^{\rm PEQ}\|_0=1$, such that
\begin{equation}\label{deigPEQ}
(\nabla_h u_h^{\rm PEQ},v_h) = \lambda_h^{\rm PEQ}(\Pi_0u_h^{\rm PEQ},v_h) = 0,~\forall v_h\in V_h^{\rm EQ}.
\end{equation}
We assert the two elements are equivalence by judging the resolvent operators $T_h^{\rm RRT}$ and $T_h^{\rm PEQ}$, which always produce the same result, implying $T_h^{\rm RRT}=T_h^{\rm PEQ}$. In fact, \eqref{deigmp} solves the eigenpairs of $T_h^{\rm RRT}$ and \eqref{deigPEQ} of $T_h^{\rm PEQ}$.
\begin{corollary}
Let $(\lambda_h^{\rm RRT},\bo{\sigma}_h^{\rm RRT},u_h^{\rm RRT})$ and $(\lambda_h^{\rm PEQ},u_h^{\rm PEQ})$ be the eigenpairs of \eqref{deigmp} and \eqref{deigPEQ}, respectively. Then we have
\begin{equation}
\lambda_h^{\rm RRT} = \lambda_h^{\rm PEQ},~\bo{\sigma}_h^{\rm RRT} = -\nabla_h u_h^{\rm PEQ},~u_h^{\rm RRT} = \Pi_0u_h^{\rm PEQ}.
\end{equation}
\end{corollary}

\begin{remark}
Though we only prove the equivalence between the two elements in 2-dimension, the same thing is true for problems of higher dimensions, and the proof is the same line.
\end{remark}

\begin{remark}
It is known that the standard ${\rm E}Q_1^{\rm rot}$ element produces lower bounds for Laplacian eigenvalues \cite{Li.Y2008,lin08}. It is interesting to see that the projected ${\rm E}Q_1^{\rm rot}$ element produces upper bounds for Laplacian eigenvalues on rectangle domain, since $\lambda_h^{\rm PEQ} = \lambda_h^{\rm RRT} \geq \lambda$.
\end{remark}

\section{Numerical Experiments}
\label{sec:num}

In this section, we show some numerical experiments on uniform and non-uniform meshes to verify our expansion identity. We put $\Omega=[0,\pi]^2$ in this section. Let $\T_{h,0}$ be the initial mesh grided by node vectors $X_0$ and $Y_0$. Generally the $k$-th level uniformly refined mesh $\T_{h,k}$ is generated by refining the node vectors $X_{k-1}$ and $Y_{k-1}$ by adding midpoints between each two consecutive nodes into the node vectors.
Here we give three groups of settings:
\begin{itemize}
\item[a)] Uniform mesh. $X_0 = Y_0 = \{\frac{i\pi}{8}\}_{i=0}^8$.
\item[b)] Quasi-uniform mesh. $X_0 = \{\frac{i\pi}{8}\}_{i=0}^8$, $Y_0 = \{\frac{j\pi}{16}\}_{j=0}^{16}$.
\item[c)] Nonuniform mesh. $X_0 = \{0,\frac{\pi}{4},\frac{\pi}{2},\frac{2\pi}{3},\frac{5\pi}{6},\pi\}$, $Y_0 = \{0,\frac{\pi}{6},\frac{\pi}{3},\frac{\pi}{2},\frac{3\pi}{4},\pi\}$.
\end{itemize}
Denote the residual of the expansion identity \eqref{result2} as
$r_h = e_h^1-e_h^2$ with $$e_h^1 = \lambda_h-\lambda$$ and $$e_h^2 = \frac{1}{12}\sum_{K_{i,j}\in \T_h}\Big(h_{x_i}^2\int_{K_{i,j}}u_{xx}^2\dd x\dd y + h_{y_j}^2\int_{K_{i,j}}u_{yy}^2\Big)\dd x\dd y.$$

The convergence rates of eigenvalues and residuals are computed by
$$\log_2\frac{\lambda_{2h}-\lambda}{\lambda_h-\lambda}~{\rm and}~\log_2\frac{r_{2h}}{r_h}$$
on the finest mesh $\T_h$, respectively.

We show the the values and residuals of the first six eigenvalues in Tables \ref{casea}-\ref{casec:residual}, where residuals of simple eigenvalues are considered when mesh is not uniform. To make things clearer, we display the vales of $e_h^1$ and $e_h^2$ for the first eigenvalue in each case in Figures \ref{Fig1}-\ref{Fig3}, where $e_h^2$ in cases a) and b) are computed exactly while in cases c) $e_h^2$ is computed numerically with high-order numerical integral.

\begin{table}
\centering
\caption{The first six eigenvalues on uniform mesh a)}\label{casea}
\begin{tabular}{cccccccc}
\toprule
\ &$\T_{h,0}$&$\T_{h,1}$&$\T_{h,2}$&$\T_{h,3}$&$\T_{h,4}$&Trend&Rate\\
\midrule
$\lambda_{1,h}$&2.0258&2.0064&2.0016&2.0004&2.0001&$\searrow$&2.00\\
$\lambda_{2,h}$&5.2225&5.0549&5.0137&5.0034&5.0009&$\searrow$&2.00\\
$\lambda_{3,h}$&5.2225&5.0549&5.0137&5.0034&5.0009&$\searrow$&2.00\\
$\lambda_{4,h}$&8.4191&8.1033&8.0257&8.0064&8.0016&$\searrow$&2.00\\
$\lambda_{5,h}$&11.0932&10.2663&10.0660&10.0165&10.0041&$\searrow$&2.00\\
$\lambda_{6,h}$&11.0932&10.2663&10.0660&10.0165&10.0041&$\searrow$&2.00\\
\botrule
\end{tabular}
\end{table}

For each column of Table \ref{casea}, the values of $\lambda_{2,h}$ and $\lambda_{3,h}$ are equal up to tolerance error $10^{-12}$. The same thing is true for $\lambda_{5,h}$ and $\lambda_{6,h}$.

\begin{table}
\centering
\caption{Residuals of the first six eigenvalues on uniform mesh a)}\label{casea:residual}
\begin{tabular}{ccccccc}
\toprule
\ &$\T_{h,0}$&$\T_{h,1}$&$\T_{h,2}$&$\T_{h,3}$&$\T_{h,4}$&Rate\\
\midrule
$r_{1,h}$&1.30e-04&8.23e-06&5.16e-07&3.22e-08&2.02e-09&4.00\\
$r_{2,h}$&4.00e-03&2.64e-04&1.67e-05&1.05e-06&6.55e-08&4.00\\
$r_{3,h}$&4.00e-03&2.64e-04&1.67e-05&1.05e-06&6.55e-08&4.00\\
$r_{4,h}$&7.86e-03&5.20e-04&3.29e-05&2.06e-06&1.29e-07&4.00\\
$r_{5,h}$&3.94e-02&2.90e-03&1.87e-04&1.17e-05&7.35e-07&4.00\\
$r_{6,h}$&3.94e-02&2.90e-03&1.87e-04&1.17e-05&7.35e-07&4.00\\
\botrule
\end{tabular}
\end{table}

\begin{figure}
  \centering
  \includegraphics[width=.75\textwidth]{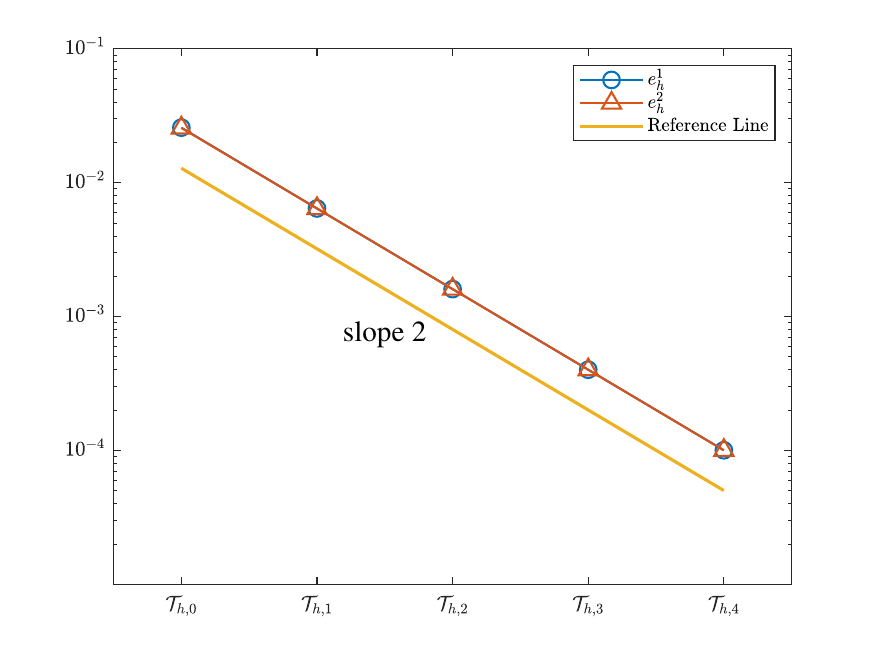}
  \caption{Error and the dominant term in error expansion of the first eigenvalue in case a)}\label{Fig1}
\end{figure}


\begin{table}
\centering
\caption{The first six eigenvalues on quasi-uniform mesh b)}\label{caseb}
\begin{tabular}{cccccccc}
\toprule
\ &$\T_{h,0}$&$\T_{h,1}$&$\T_{h,2}$&$\T_{h,3}$&$\T_{h,4}$&Trend&Rate\\
\midrule
$\lambda_{1,h}$&2.0161&2.0040&2.0010&2.0003&2.0001&$\searrow$&2.00\\
$\lambda_{2,h}$&5.0646&5.0161&5.0040&5.0010&5.0003&$\searrow$&2.00\\
$\lambda_{3,h}$&5.2128&5.0525&5.0131&5.0033&5.0008&$\searrow$&2.00\\
$\lambda_{4,h}$&8.2612&8.0645&8.0161&8.0040&8.0010&$\searrow$&2.00\\
$\lambda_{5,h}$&10.2760&10.0685&10.0171&10.0043&10.0011&$\searrow$&2.00\\
$\lambda_{6,h}$&11.0835&10.2639&10.0654&10.0163&10.0041&$\searrow$&2.00\\
\botrule
\end{tabular}
\end{table}

\begin{table}
\centering
\caption{Residuals of the first three simple eigenvalues on quasi-uniform mesh b)}\label{caseb:residual}
\begin{tabular}{ccccccc}
\toprule
\ &$\T_{h,0}$&$\T_{h,1}$&$\T_{h,2}$&$\T_{h,3}$&$\T_{h,4}$&Rate\\
\midrule
$r_{1,h}$&6.91e-05&4.37e-06&2.74e-07&1.71e-08&1.07e-09&4.00\\
$r_{4,h}$&4.19e-03&2.76e-04&1.75e-05&1.10e-06&6.85e-07&4.00\\
$r_{11,h}$&1.16e-00&3.08e-03&1.98e-04&1.25e-05&7.80e-07&4.00\\
\botrule
\end{tabular}
\end{table}

\begin{figure}
  \centering
  \includegraphics[width=.75\textwidth]{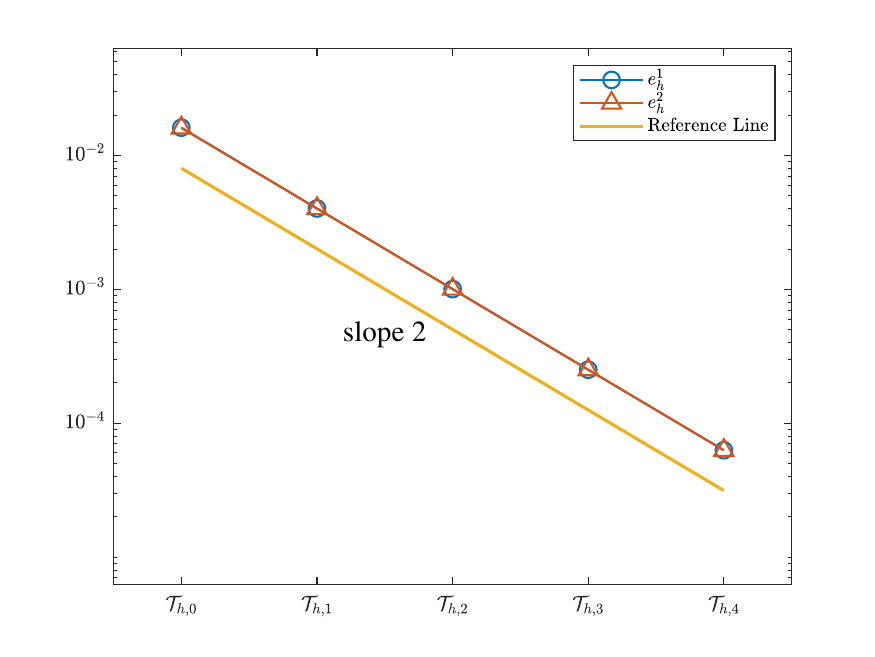}
  \caption{Error and the dominant term in error expansion of the first eigenvalue in case b)}\label{Fig2}
\end{figure}


\begin{table}
\centering
\caption{Eigenvalues on nonuniform mesh c)}\label{casec}
\begin{tabular}{cccccccc}
\toprule
\ &$\T_{h,0}$&$\T_{h,1}$&$\T_{h,2}$&$\T_{h,3}$&$\T_{h,4}$&Trend&Rate\\
\midrule
$\lambda_{1,h}$&2.0750&2.0186&2.0046&2.0012&2.0003&$\searrow$&2.00\\
$\lambda_{2,h}$&5.6299&5.1583&5.0395&5.0099&5.0025&$\searrow$&2.00\\
$\lambda_{3,h}$&5.6299&5.1583&5.0395&5.0099&5.0025&$\searrow$&2.00\\
$\lambda_{4,h}$&9.1848&8.2981&8.0743&8.0186&8.0046&$\searrow$&2.00\\
$\lambda_{5,h}$&13.0390&10.7760&10.1913&10.0476&10.0119&$\searrow$&2.00\\
$\lambda_{6,h}$&13.0390&10.7760&10.1913&10.0476&10.0119&$\searrow$&2.00\\
\botrule
\end{tabular}
\end{table}

\begin{table}
\centering
\caption{Residuals of the first three simple eigenvalues on non-uniform mesh c)}\label{casec:residual}
\begin{tabular}{ccccccc}
\toprule
\ &$\T_{h,0}$&$\T_{h,1}$&$\T_{h,2}$&$\T_{h,3}$&$\T_{h,4}$&Rate\\
\midrule
$r_{1,h}$&7.16e-04&4.67e-05&2.94e-06&1.84e-07&1.15e-08&4.00\\
$r_{4,h}$&-3.16e-03&1.06e-03&8.20e-05&5.36e-06&3.39e-07&4.00\\
$r_{11,h}$&-1.12e-02&2.99e-02&2.09e-03&1.34e-04&8.40e-06&4.00\\
\botrule
\end{tabular}
\end{table}

\begin{figure}
  \centering
  \includegraphics[width=.75\textwidth]{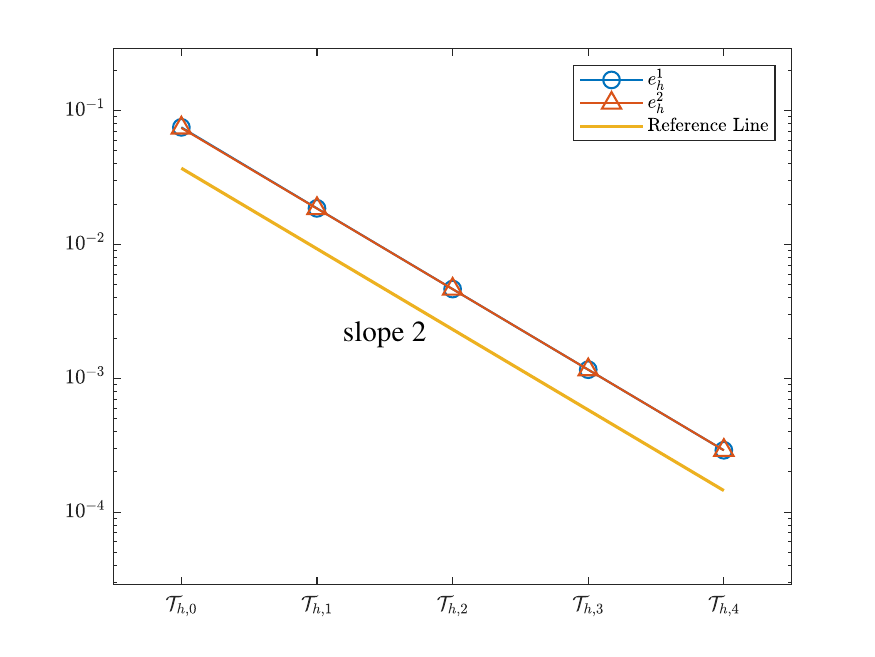}
  \caption{Error and the dominant term in error expansion of the first eigenvalue in case c)}\label{Fig3}
\end{figure}

\FloatBarrier

\begin{appendices}
\renewcommand{\theHequation}{appendix.\arabic{equation}}

\section{Proof of Theorem \ref{thm:multiple}}
\label{secA1}
\begin{lemma}\label{space}
Let $X_*$ and $Y_*$ be proper subspaces of Hilbert spaces $X$ and $Y$ endowed inner product $(\cdot,\cdot)_H$ with $\delta(X,Y)=\epsilon$ and $\delta(X_*,Y_*)\leq c\epsilon$, where $\epsilon$ is sufficiently small and $c$ is independent of $\epsilon$. Then
\begin{equation}
\delta(X_*^{\perp},Y_*^{\perp}) \leq C\epsilon,
\end{equation}
where $X_*^{\perp}$ and $Y_*^{\perp}$ are the orthogonal complementary spaces of $X_*$ and $Y_*$ in $X$ and $Y$ under $(\cdot,\cdot)_H$, respectively, and $C$ depends on $c$ but not on $\epsilon$.
\end{lemma}
\begin{proof}
The property of orthogonal projection gives $P_X = P_{X_*} + P_{X_*^{\perp}}$ and $P_Y = P_{Y_*} + P_{Y_*^{\perp}}$. Hence we have
$$\delta(X_*^{\perp},Y_*^{\perp}) = \|P_{X_*^{\perp}}-P_{Y_*^{\perp}}\|\leq \|P_X-P_Y\|+\|P_{X_*}-P_{Y_*}\|\leq (1+c)\epsilon,$$
where the operator norm $\|\cdot\|$ is induced by $(\cdot,\cdot)_H$.
\end{proof}

\begin{lemma}
Denote  $\bo{\Sigma}_{h0}^{\perp}=\{\bs_h\in\bo{\Sigma}_h:(\bs_h,\bt_h)=0,~\forall \bt_h\in\bo{\Sigma}_{h0}\}$. The eigenvalue problem \eqref{deigmp} is equivalent to the $H(\ddiv)$ formulation: find $(\bar{\lambda}_h,\bar{\bs}_h)\in\mathbb{R}\times\bo{\Sigma}_{h0}^{\perp}$ with $\|\bar{\bs}_h\|_0=1$ such that
\begin{equation}\label{deigmpdiv}
(\ddiv\bar{\bs}_h,\ddiv\bar{\bs}_h) = \bar{\lambda}_h(\bar{\bs}_h,\bar{\bs}_h),
\end{equation}
with relations
\begin{equation}
\bar\lambda_h = \lambda_h,~\bar\bs_h = \lambda_h^{-1/2}\bs_h,~u_h = \lambda_h^{-1/2}\ddiv\bar\bs_h.
\end{equation}
\end{lemma}
\begin{proof}
The formulation \eqref{deigmpdiv} is nothing but eliminating $u_h$ from the mixed problem \eqref{deigmp} and normalizing $\bs_h$. We do not distinguish $\bar\lambda_h$ from $\lambda_h$ in the following text.
\end{proof}
Define the Rayleigh quotient as
$$R(\bt_h):=\frac{(\ddiv\bt_h,\ddiv\bt_h)}{(\bt_h,\bt_h)},~\forall \bt_h\in\bo{\Sigma}_{h0}^{\perp}\backslash\{\bo{0}\}.$$

\begin{lemma}\label{min-max}
The min-max principle gives
$$\bar{\lambda}_{k,h} = \min\limits_{\substack{\bo{X}_h^k\subset\bo{\Sigma}_{h0}^{\perp},\\\dim \bo{X}_h^k=k}}\max\limits_{\substack{\bt_h\in\bo{X}_h^k,\\\bt_h\neq\bo{0}}}R(\bt_h),$$
and
$$\bar{\lambda}_{k,h} = \max\limits_{\substack{\bo{X}_h^{k-1}\subset\bo{\Sigma}_{h0}^{\perp},\\\dim \bo{X}_h^{k-1}=k-1}}\min\limits_{\substack{\bt_h\in\left(\bo{\Sigma}_{h0}\oplus\bo{X}_h^{k-1}\right)^{\perp},\\\bt_h\neq\bo{0}}}R(\bt_h).$$
\end{lemma}

\begin{lemma}\label{RQProperty}
Take $\bt_h = \sum_{k}\beta_k\bs_{k,h}$, where $\bs_{k,h}$ is the eigenfunctions of \eqref{deigmp} and $\sum_{k}\beta_k^2=1$, then
$$R(\bt_h) = \sum_{k}\beta_k^2R(\bs_{k,h}).$$
\end{lemma}

\begin{lemma}\label{RQIntp}
Let $\bs_I$ be the interpolation of the k-th eigenfunction $\bs_k$ with $\bs_k\in\bo{H}^4(\OO)$, then
\begin{equation}
R(\bs_I)-\lambda_k = (\bs_k-\bs_I,\bs_I) + O(h^4).
\end{equation} 
\end{lemma}

\begin{proof}
Direct computation gives
\begin{equation}
\begin{aligned}
R(\bs_I)-\lambda_k &= \frac{(\ddiv\bs_I,\ddiv\bs_I)-\lambda_k(\bs_I,\bs_I)}{(\bs_I,\bs_I)}\\
&= \frac{\lambda_k}{(\bs_I,\bs_I)}((\ddiv\bs_I,u_k)-(\bs_I,\bs_I))\\
&= \frac{\lambda_k}{(\bs_I,\bs_I)}(\bs_k-\bs_I,\bs_I),
\end{aligned}
\end{equation}
and by applying \eqref{errintp} and \eqref{explem}, we can see
$$\lambda_k - (\bs_I,\bs_I) = (\bs_k-\bs_I,\bs_k+\bs_I) = \|\bs_k-\bs_I\|_0^2 + 2(\bs_k-\bs_I,\bs_I) = O(h^2),$$
which completes the proof.
\end{proof}

Recall that $\Lambda_k$ is the index set $\{j:\lambda_j = \lambda_k\}$. Suppose we have all decomposition pairs $\{(m_i,n_i)\}_{i=1}^p$ with $m_i^2+n_i^2 = \lambda_k$, where we recall each pair $(m_i,n_i)$ is a \enquote{frequency} and we do not distinguish the order of $m_i$ and $n_i$. The pairs are sorted in ascending order of $m_i^4+n_i^4$. Each \enquote{frequency} identities an eigenspace $E_k^i = \spn\{u_{m_i,n_i},u_{n_i,m_i}\}$. Basic inequalities give $m_i^4+n_i^4>m_j^4+n_j^4$ iff $|m_i-n_i|>|m_j-n_j|$. Then the eigenspace for $\lambda_k$ can be rewritten as $E_k=\bigoplus_{i=1}^p E_k^i$. We denote the corresponding eigenspace of discrete problem \eqref{deigmp} as $E_{k,h}$. The standard spectral theory gives $\delta(E_k,E_{k,h}) = O(h)$, and thus $\dim(E_k)=\dim(E_{k,h}) = N_k$. If $N_k$ is odd, then $N_k = 2p-1$, otherwise $N_k=2p$. There is an orthogonal decomposition $E_{k,h}=\bigoplus_{i=1}^p E_{k,h}^i$ where the discrete eigenpairs are sorted in ascending order of eigenvalues and then the eigenfunctions corresponding to pair $(m_i,n_i)$ spans $E_{k,h}^i$ with $\dim(E_{k,h}^i) = \dim(E_k^i)$. 

\begin{lemma}\label{lem:Spaceconv}
Given $\OO=[0,\pi]^2$ and the mesh is uniform. Then there exists an eigenspace $E_{k,h}^1\subset E_{k,h}$ such that $\delta(E_{k,h}^1,E_k^1) = O(h)$.
\end{lemma}
\begin{proof}
Let $\lambda_{k,h}$ be the lowest eigenvalue of problem \eqref{deigmp} which converges to $\lambda_k$ and $(\bs_{k,h},u_{k,h})$ be the corresponding eigenfunctions. Denote the index set $\Lambda_{k^-} = \{j:\lambda_j\leq\lambda_k\}$. Let $\bo{F}_{k^-,h}$ be the summation of all the $\bo{F}_{j,h}$ with $j\in\Lambda_{k^-}$, and $\bo{F}_{k^-,h}^{\perp}$ be the orthogonal complementary space of $\bo{F}_{k^-,h}$ in $\bo{\Sigma}_{h0}^{\perp}$, and then from Lemma \ref{min-max}, $\bs_{k,h}$ minimize $R(\bt_h)$ in $\bo{F}_{k^-,h}^{\perp}$ and $R(\bs_{k,h}) = \lambda_{k,h}$.
 
\textbf{Step A:} We shall provide a lower bound for $\lambda_{k,h}$:
\begin{equation*}
\lambda_{k,h}-\lambda_k\geq\frac{m_1^4+n_1^4}{12}h^2+O(h^4).
\end{equation*}

From the convergence of eigenspace, for each $\T_h$ there exists $u\in\bigoplus_{i=1}^p E_k^i$ such that \eqref{result2} holds and
\begin{equation}
\|u_{k,h}-u\|_0=O(h)
\end{equation}
with $\|u\|_0=1$. Decompose $u$ as $u = \sum_{i=1}^p\alpha_i(h)u^i$ such that $u^i\in E_k^i$ with $\|u^i\|_0=1$ as well as $\sum_{i=1}^p\alpha_i^2(h)=1$, and the decomposition is unique. We omit $h$ in $\alpha_i(h)$ for simplicity hereinafter. Thus from \eqref{result2} we have
\begin{equation}\label{StepAbelow}
\lambda_{k,h} - \lambda_k = \frac{h^2}{12}\sum_{i=1}^p(m_i^4+ n_i^4)\alpha_i^2 + O(h^4)\geq \frac{h^2}{12}(m_1^4+ n_1^4) + O(h^4).
\end{equation}

\textbf{Step B:} We shall show when $d(u_{k,h},E_k^1) = O(h)$ is not true, where $d(u_{k,h},E_k^1)$ is the distance from $u_{k,h}$ to $E_k^1$, then
\begin{equation*}
\lambda_{k,h} - \lambda_k = \frac{h^2}{12}\sum_{i=1}^p(m_i^4+ n_i^4)\alpha_i^2 + O(h^4)> \frac{h^2}{12}(m_1^4+ n_1^4) + O(h^4),
\end{equation*}
i.e., the inequality in \eqref{StepAbelow} is strict. 

In fact, if we assume there is a subsequence of $\{u_{k,h}\}$ (still denoted as $\{u_{k,h}\}$) such that $\lim_{h\rightarrow0}h^{-1}d(u_{k,h},E_k^1)=+\infty$, then $d(u_{k,h},E_k^1)\leq \|u_{k,h}-u\| + d(u,E_k^1)$ implies
\begin{equation}
\lim_{h\rightarrow0}h^{-1}\max_{1<i\leq q}\{|\alpha_i|\}= +\infty.
\end{equation}
Thus the equation in \eqref{StepAbelow} cannot hold.

\textbf{Step C:} We shall choose $\bt_h\in\bo{F}_{k^-,h}^{\perp}$ such that
\begin{equation*}
R(\bt_h)-\lambda_k = \frac{m_1^4+n_1^4}{12}h^2+O(h^4).
\end{equation*}

Let $u$ be a unit eigenfunction in $E_k^1$ and $\bs_I$ be the Raviart-Thomas interpolation of $\bs=-\nabla u$. From the proof of Theorem \ref{thm:delicateExpansion}, we have
\begin{equation}
(\bs-\bs_I,\bs_I) = \frac{h^2}{12}\sum_{K_{i,j}}\int_{\OO}u_{xx}^2 + u_{yy}^2\dd x\dd y  + O(h^4) = \frac{h^2}{12}(m_1^4+ n_1^4) + O(h^4),
\end{equation}
and applying Lemma \ref{RQIntp}, we can obtain
\begin{equation}
R(\bs_I)-\lambda_k = \frac{h^2}{12}(m_1^4+ n_1^4) + O(h^4).
\end{equation}
Put $\bar\bs_I = \bs_I/\|\bs_I\|_0$ and decompose $\bar\bs_I$ into $\bar\bs_I = \bar\bs_I^1 + \bar\bs_I^2$ with $(\bar\bs_I^1,\bar\bs_I^2)\in\bo{F}_{k^-,h}\times\bo{F}_{k^-,h}^{\perp}$. Thus from Lemma \ref{RQProperty} we have
\begin{equation}\label{StepCRQD}
R(\bar\bs_I)= \|\bar\bs_I^1\|_0^2R(\bar\bs_I^1) + (1-\|\bar\bs_I^1\|_0^2)R(\bar\bs_I^2).
\end{equation}
Note that for each $\bs_{j,h}\in\bo{F}_{k^-,h}$, we have
\begin{equation}
(\bs_I,\bs_{j,h}) = (\bs_I-\bs,\bs_{j,h}) + (\bs-\bs_I,\bs_{j,h}-\bs_j) + (\bs_I,\bs_{j,h}-\bs_{j,I}) + (\bs_I,\bs_{j,I}-\bs_j)= O(h^2)
\end{equation}
from \eqref{errsigma}, Theorem \ref{thm:integralExpansion}, \eqref{supcs}, and the Cauchy-Schwarz inequality, which implies $\|\bar\bs_I^1\|_0 = O(h^2)$. Then we can derive from \eqref{StepCRQD} that
$$R(\bar\bs_I^2) = (1+O(h^4))R(\bs_I) = \lambda_k + \frac{h^2}{12}(m_1^4+ n_1^4) + O(h^4),$$
which means $\bt_h=\bar\bs_I^2$ is what we desire.

If $m_1=n_1$, the proof is completed. Otherwise, we go to Step D.

\textbf{Step D:} We shall find another $\bt_h^2\in\bo{F}_{k^-,h}^{\perp}$ with $(\bt_h^2,\bs_{k,h})=0$ such that
\begin{equation*}
R(\bt_h^2) = \lambda_k + \frac{h^2}{12}(m_1^4+ n_1^4) + O(h^4).
\end{equation*}  

Since $m_1\neq n_1$, $\dim(E_k^1)=2$. From Step B we can choose $u\in E_k^1$ such that $\|u-u_{k,h}\|_0 = O(h)$ and the procedure in Step C is valid. Now let $\tilde{u}$ be another unit eigenfunction in $E_k^1$ orthogonal to $u$ in $L^2$ norm. Then the same procedure in Step C can be done: Put $\bar{\tilde\bs}_I = \tilde\bs_I/\|\tilde\bs_I\|_0$ and decompose $\bar{\tilde\bs}_I$ into $\bar{\tilde\bs}_I = \bar{\tilde\bs}_I^1 + \bar{\tilde\bs}_I^2$ with $(\bar{\tilde\bs}_I^1,\bar{\tilde\bs}_I^2)\in\bo{F}_{k^-,h}\times\bo{F}_{k^-,h}^{\perp}$ and then $\|\bar{\tilde\bs}_I^1\|=O(h^2)$ and
$$R(\bar{\tilde\bs}_I^2) = \lambda_k + \frac{h^2}{12}(m_1^4+ n_1^4) + O(h^4).$$
Now decompose $\bar{\tilde\bs}_I^2$ as $\bar{\tilde\bs}_I^2 = \bt_h^1+\bt_h^2$ with $\bt_h^1 = \lambda_{k,h}^{-1}(\bar{\tilde\bs}_I^2,\bs_{k,h})\bs_{k,h}$ and $\bt_h^2 = \bar{\tilde\bs}_I^2-\bt_h^1$. Since $$R(\bt_h^1) = R(\bs_{k,h}) = \lambda_k + \frac{h^2}{12}(m_1^4+ n_1^4) + O(h^4),$$
we can derive from Lemma \ref{RQProperty} that
$$R(\bt_h^2) = \lambda_k + \frac{h^2}{12}(m_1^4+ n_1^4) + O(h^4).$$
Combining Lemma \ref{min-max} it indicates 
$$\lambda_{k+1,h} - \lambda_k \leq \frac{h^2}{12}(m_1^4+ n_1^4) + O(h^4),$$
and then from Step B we have $d(u_{k+1},E_{k,h}^1)=O(h)$, which completes the proof.
\end{proof}

\begin{lemma}\label{lem:Spaceconv2}
Given $\OO=[0,1]^2$ and the mesh is uniform. Then $E_{k,h} = \bigoplus_{i=1}^p E_{k,h}^i$ and $\delta(E_{k,h}^q,E_k^q) = O(h)$ for $q = 1,2,\ldots,p$.
\end{lemma}
\begin{proof}
We shall apply Lemma \ref{lem:Spaceconv} by induction. Suppose $\delta(E_{k,h}^q,E_k^q) = O(h)$ for $q=1,2,\ldots,t$. The details of proving $\delta(E_{k,h}^{t+1},E_k^{t+1}) = O(h)$ are analogous to that of Lemma \ref{lem:Spaceconv}. From Lemma \ref{space}, we have $$\delta\left(\bigoplus_{i=t+1}^p E_k^i,\bigoplus_{i=t+1}^p E_{k,h}^i\right)=O(h).$$
Thus all steps in the proof of Lemma \ref{lem:Spaceconv} can be done analogously, except for a mild modification in Step C. 

\textbf{Modification in Step C:} For a unit eigenfunction $u\in E_k^{t+1}$, let $\bs = -\nabla u$ and $\bs_I$ be the Raviart-Thomas interpolation of $\bs$. Put $\bar\bs_I = \bs_I/\|\bs_I\|_0$. We decompose $\bar\bs_I$ into $\bar\bs_I = \bar\bs_I^1 + \bar\bs_I^2 + \bar\bs_I^3$ with $$(\bar\bs_I^1,\bar\bs_I^2,\bar\bs_I^3)\in\bo{F}_{k^-,h}\times\bigoplus_{i=1}^t E_{k,h}^i\times\left(\bigoplus_{i=1}^t\bo{F}_{k^-,h}\oplus E_{k,h}^i\right)^{\perp}.$$
We have $\|\bar\bs_I^2\|_0 = O(h)$ since for each $u_{j,h}\in \bigoplus_{i=1}^t E_{k,h}^i$,
$$(\bs_I,\bs_{j,h}) = (\bs_I-\bs,\bs_{j,h}) + (\bs,\bs_{j,h}-\bs_j)=O(h).$$
Now that $\|\bar\bs_I^1\|_0 = O(h^2)$ (the same as Lemma \ref{lem:Spaceconv}), $\|\bar\bs_I^2\|_0 = O(h)$ and $R(\bar\bs_I) = \lambda_k + \frac{h^2}{12}(m_{q+1}^4+ n_{q+1}^4) + O(h^2)$, which indicate
$$R(\bar\bs_I^3) = \lambda_k + \frac{h^2}{12}(m_{q+1}^4+ n_{q+1}^4) + O(h^4).$$
\end{proof}

A combination of \eqref{result2} and Lemma \ref{lem:Spaceconv2} completes the proof of Theorem \ref{thm:multiple}. 

\end{appendices}

\backmatter

\section*{Declarations}
\bmhead{Conflict of Interest}
The authors declare that they have no conflict of interest.

\bmhead{Funding}
This paper is supported by Strategic Priority Research Program of Chinese Academy of Sciences (XDB 0640000), National Key R\&D Program of China (No. 2024YFA1012503 and 2022YFA1004500), National Natural Science Foundation of China (No. 12371372 and 12271512), and Natural Science Foundation of Fujian Province (No. 2025J01026).

\bmhead{Author Contributions}
\textbf{Yifan Yue:} Writing--original draft, Writing--reviewing $\&$ editing, Methodology, Visualization, Validation, Software; \textbf{Hongtao Chen:} Writing--review $\&$ editing, Conceptualization, Methodology, Supervision, Project administration, Funding acquisition; \textbf{Shuo Zhang:} Writing – review $\&$ editing, Conceptualization, Methodology, Supervision, Funding acquisition.

\bmhead{Acknowledgements}
Not applicable.



\end{document}